\documentclass[11pt]{article}
\usepackage{amsmath,amsthm,amsfonts,amssymb,epsfig,bm}

\usepackage{lipsum}
\usepackage[usenames]{color}
\usepackage[colorlinks=true]{hyperref}
\usepackage{graphicx}
\usepackage[scale=.8]{geometry}
\usepackage{ulem}
\usepackage{textcomp}
\usepackage{amsmath}
\usepackage{esint}
\usepackage{cleveref}
\usepackage[export]{adjustbox} 
\usepackage{caption}
\usepackage{subcaption}
\usepackage{authblk}

\parskip        \medskipamount

\newtheorem{theorem}{Theorem}

\newtheorem{lemma}{Lemma}

\newtheorem{proposition}{Proposition}
\newtheorem{definition}{Definition}

\newtheorem{question}{Question}
\theoremstyle{remark}
\newtheorem{remark}{Remark}
\newtheorem{corollary}{Corollary}

\newtheorem{example}{Example}
\newtheorem{hypothesis}{Hypothesis}

\def\beqlb{\begin{eqnarray}}\def\eeqlb{\end{eqnarray}}
\def\beqnn{\begin{eqnarray*}}\def\eeqnn{\end{eqnarray*}}

\def\N{\mathbb{N}}

\def\R{\mathbb{R}}

\def\P{\mathbb{P}}
\def\E{\mathbb{E}}

\def\TT{\mathcal{T}}

\renewcommand{\Phi}{\varPhi}
\renewcommand{\epsilon}{\varepsilon}

\newcommand{\var}{\operatorname{var}}

\renewcommand{\d}{\text{\rm\,d}}

\definecolor{mygray}{gray}{0.9}
\definecolor{deeppink}{RGB}{255,20,147}
\definecolor{mygreen}{rgb}{0.05, 0.576, 0.03}
\definecolor{myred}{rgb}{0.768, 0.09, 0.09}

\long\def\symbolfootnote[#1]#2{\begingroup
	\def\thefootnote{\fnsymbol{footnote}}\footnote[#1]{#2}\endgroup}

\newcommand{\Ind}[1]{\mathbf{1}_{\left\{#1\right\}}}

\newcommand{\x}[1]{[x^{#1}]}
\renewcommand{\d}{{\mathrm{d}}}
\renewcommand{\L}{\mcl{L}}

\newcommand{\mcl}{\mathcal}

\newcommand{\Ll}{\left}
\newcommand{\Rr}{\right}
\newcommand{\bracket}[1]{\left\langle{#1}\right\rangle}
\newcommand{\norm}[1]{\left\Vert{#1}\right\Vert}
\newcommand{\e}{\mathbf e}

\newcommand{\G}{\mcl G}

\usepackage{MnSymbol}

\usepackage{authblk}
\begin{document}
	
	\title{\bf Size distribution of clusters in site-percolation on random recursive tree}
	\author{Chenlin Gu\thanks{Yau Mathematical Sciences Center, Tsinghua University}, \, Linglong Yuan\thanks{Department of Mathematical Sciences, University of Liverpool}}

	\maketitle
	\begin{abstract}
		We prove rigorously several results about the site-percolation on random recursive trees, observed in the previous work by Kalay and Ben-Naim (\textit{J.Phys.A}, 2015). For a random recursive tree of size $n$, let every site have probability ${p \in (0,1)}$ to remain and with probability $(1-p)$ to be removed. As $n\to\infty,$ we show that the proportion of the remaining clusters of size $k$ is of order $k^{-1-\frac{1}{p}}$, resulting in a Yule--Simon distribution; the largest cluster size is of order $n^{p}$, and admits a non-trivial scaling limit. The proofs are based on the embedding of this model in the multi-type branching processes, and a coupling with the bond-percolation on random recursive trees.
	\end{abstract}
	\vspace{8pt} \noindent {\textbf{Keywords:} branching process, random recursive tree, power law, critical behaviour, Yule--Simon distribution, site-percolation, bond-percolation, Ewens sampling formula.}
	
	
	\begin{figure}[h!]
		\centering
		\includegraphics[width=0.8\textwidth]{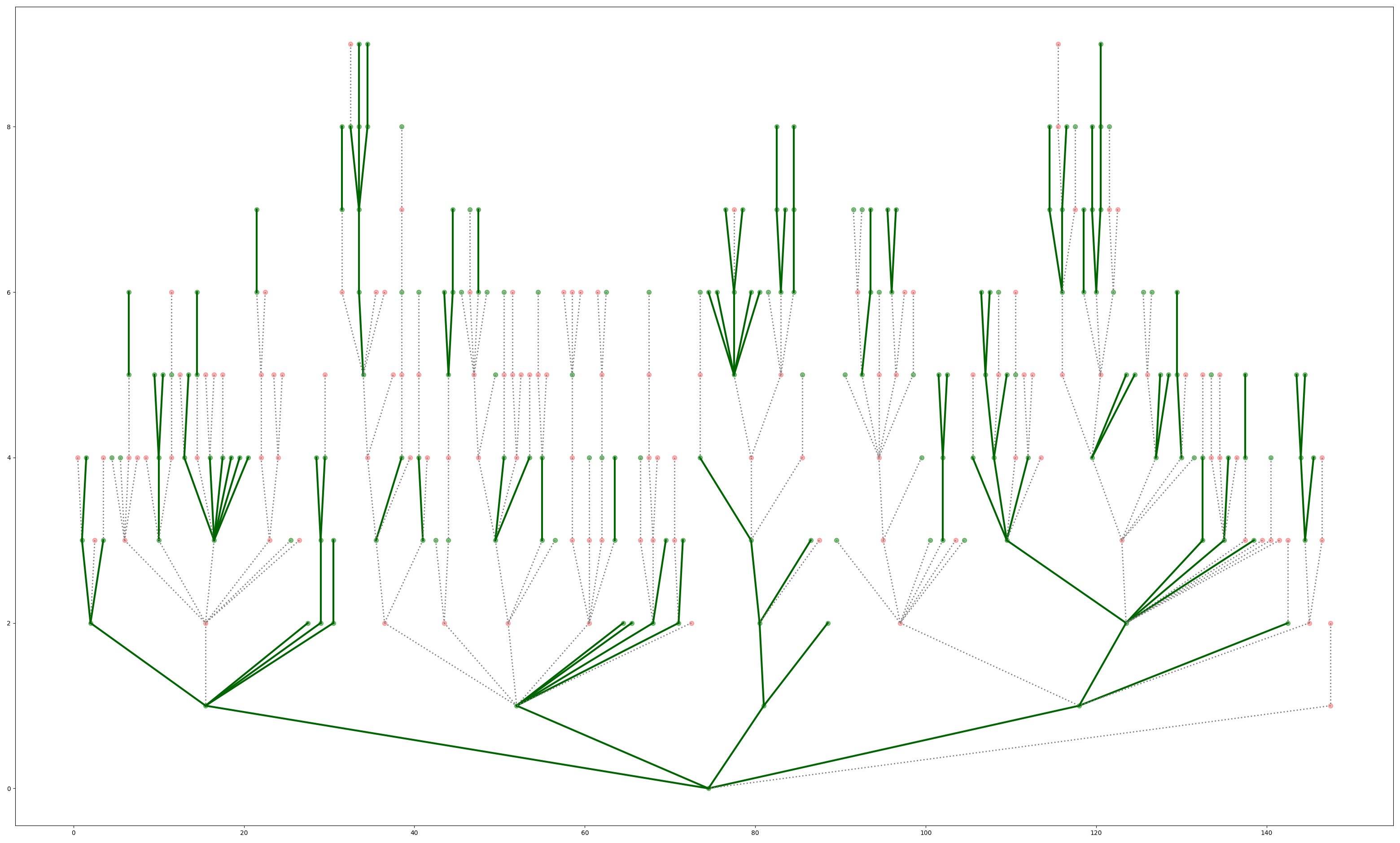}
		\caption{An illustration of the site-percolation on an RRT of size $300, p=0.6$. The sites in green are open, and the sites in red are closed.}\label{fig.RRT} 
	\end{figure}
	
	\newpage

	\section{Introduction}\label{sec:Introduction}
	\subsection{Motivation}
	The \textbf{(uniform) random recursive tree} (RRT) is a classical random object that has been studied extensively. It can be constructed by adding vertices one by one, such that every vertex chooses uniformly an existing vertex to attach. Now, for an RRT of size $n$, we implement the site-percolation in the sense that every vertex independently has probability $p \in (0,1)$ to be \textbf{open} and probability $(1-p)$ to be \textbf{closed}. 
	We remove the closed vertices and its directly attached edges, which breaks the tree into several connected components called \textbf{clusters}. See Figures~\ref{fig.RRT} and ~\ref{fig.RRT_site} for an illustration. We can verify that all clusters are independents RRTs conditioned on their sizes 
	(see Definition~\ref{def.ContinuousFrag} and Lemma~\ref{lem.RRTiid}), thus the cluster size distribution captures the key information and constitutes the topic that we will study in the present paper. 
	Especially, we are interested in two fundamental questions as $n \to \infty$:
	\begin{enumerate}
		\item What is the proportion of the clusters of size $k$ ?
		\item What is the typical size of the largest clusters ?
	\end{enumerate} 
	
	These two questions were 
	investigated first in \cite{kalay2014fragmentation} by  \footnote[1]{The model in \cite{kalay2014fragmentation} is slightly different from our description above, as their removal is implemented for $\lfloor pn \rfloor$ uniformly chosen vertices among an RRT of size $n$. However, the two models share the same large-scale behaviour, and we take the ``i.i.d." setting as a convenient option. 
	} Kalay and Ben-Naim, and they obtained the answers using heuristic arguments and Monte--Carlo simulation. 
	For the first question, they derived heuristically an infinite-dimensional differential system for the evolution of the proportion, and then observed an explicit solution $(\frac{1}{p} - 1) \frac{\Gamma(1 + \frac{1}{p})\Gamma(k)}{\Gamma(k+1+\frac{1}{p})}$ as the proportion of the clusters of size $k$ (see \cite[(15)]{kalay2014fragmentation}), which is approximately $k^{-1-\frac{1}{p}}$ for large $k$. It was not mentioned in the paper, but this explicit solution follows the Yule--Simon distribution with parameter $\frac{1}{p}$. However, the numerical approximation of this limit proportion is not stable for large $k$ (see \cite[Figure~4]{kalay2014fragmentation}), which leads to the second question of determining the typical size of the largest clusters, as raised by Kalay and Ben-Naim. 
	They claimed that the largest cluster size is of order $n^{p}$ (\cite[(20)]{kalay2014fragmentation}), as a result of applying the above obtained proportion formula and a heuristic argument of an extremal statistics criterion (see \eqref{eqn:heuristic}). 
	This paper aims to provide mathematical arguments behind the statements for the two questions by 
	Kalay and Ben-Naim, and justifies rigorously more precise versions of these results.

	Let us mention some other related previous work. There are numerous results in the literature on cutting/removing edges. In \cite{meir1974cutting}, Meir and Moon were interested in how many steps are needed to isolate a vertex when deleting edges uniformly one after another in an RRT. See \cite{drmota2009limiting, iksanov2007probabilistic, panholzer2004destruction} for asymptotic results on the number of cuts to isolate the root.   In  the bond-percolation model where every edge is removed independently with probability $(1-p)$, 
	Bauer and Bertoin studied in \cite{bertoin2014sizes} the sizes of the largest clusters in the supercritical regime, 
	where the giant clusters emerge since the parameter of percolation $p(n)$, as a function of $n$, satisfies $p(n) \xrightarrow{n \to \infty}1$.  The case $p(n) \equiv p \in (0,1)$ is considered as the critical percolation on RRTs, and the size of clusters are discussed by  Bauer and Bertoin in \cite{baur2015fragmentation}. Recently, to model the pandemic of Covid-19, several researches including \cite{bertoin2022} by Bertoin, \cite{bansaye2021growth} by Bansaye and the authors of this paper, and \cite{bellin2023uniform, bellin2024uniformattachmentfreezingscaling} by Bellin, Blanc-Renaudie, Kammerer,  and Kortchemski were devoted to studying the behaviour of the size distribution in breaking and growing RRTs. The bond-percolation on RRT also has applications in elephant random walk and see \cite{kursten2016, businger2018, guerin2023fixed, qin2024}.

	\begin{figure}[t]
		\centering
		\includegraphics[width=\textwidth]{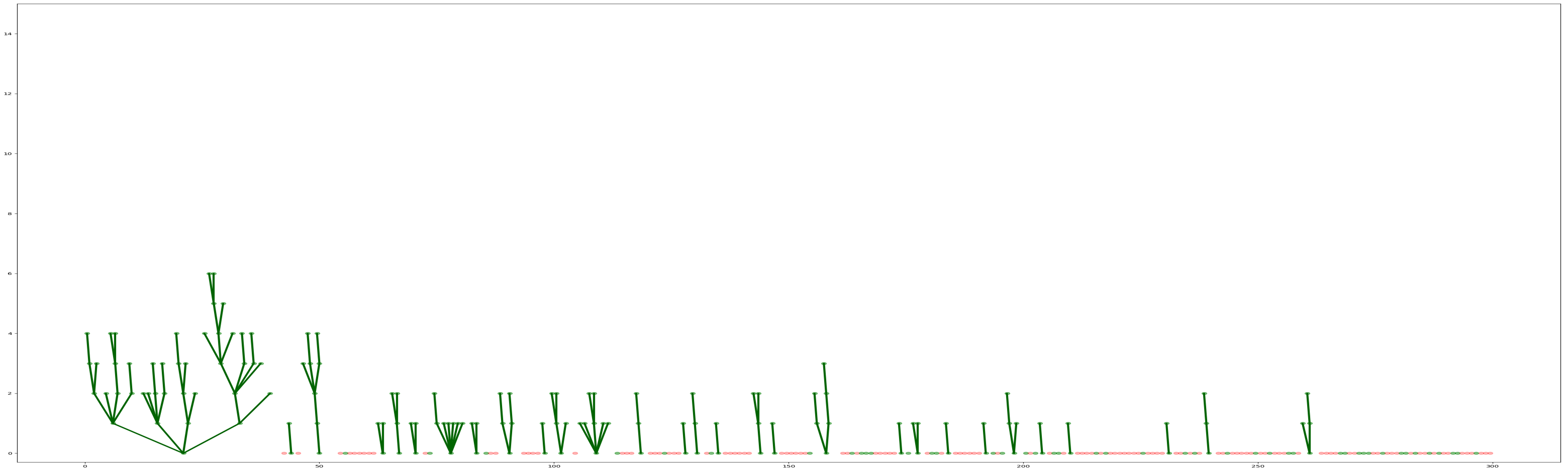}
		\caption{A decomposition of RRT in Figure~\ref{fig.RRT} into clusters.}\label{fig.RRT_site}
	\end{figure}

	In terms of cutting/removing vertices, recently  \cite{eslava2024degree} introduced a degree-biased version of cutting a vertex such that the subtree rooted at the parent vertex of a uniformly chosen edge is removed, and the number of cuts is studied to isolate the root.
	Apart from \cite{kalay2014fragmentation} by Kalay and Ben-Naim, we are not aware of any other work about the site-percolation on RRT to our best knowledge. This is partly because the site-percolation shares similar universal behaviour like the bond-percolation, while the former is a bit more complicated in nature. To see this, notice that every removal of an edge in an RRT always yields two clusters, but 
	removing a vertex can yield a random number of clusters.
	But it is known that 
	there exists a natural link between the bond-percolation and site-percolation on RRTs. 
	The other motivation of the paper is thus to explore this link fully and use it to facilitate the study of site-percolation.

	
	\subsection{Main results}\label{subsec.main}
	Given a graph $G$, we denote by $V(G)$ and $E(G)$ respectively for its vertex/site set and edge/bond set. A tree is a connected graph $T$ satisfying $\vert E(T) \vert = \vert V(T)\vert - 1$, so there exists a unique geodesic path to connect every two vertices on it. When every vertex of the tree is assigned a different label, we call it a labelled tree and identify the label set as the vertex set. In this paper, the label takes values in $\N_+ = \{1,2,3,4 \cdots\}$, then the vertex with the smallest label is called \textbf{the root}. We study especially \textbf{the recursive tree}, which is a labelled tree, such that the labels are increasing along the geodesic path from the root to every other vertex. From the construction of RRT, it is not hard to see that RRT is uniformly distributed on the set of recursive trees of fixed size. Let $p \in (0,1)$ be fixed throughout the paper. In order to couple all RRTs of different sizes and their site-percolations in a natural common space, we propose \textbf{the canonical coupling} following the convention in \cite{baur2015fragmentation}.
	\begin{definition}[Canonical coupling]\label{def.Canonical}
		We label the vertices by $\N_+ = \{1,2,3,\cdots\}$ as their order and implement the recursive construction introduced above to obtain an infinite random recursive tree $T_\infty = (\N_+, E(T_\infty))$. Let  $\omega: \N_+ \to \{0,1\}$ follow the Bernoulli product measure $\operatorname{Ber}(p)^{\otimes \N_+}$ to stand for the state of every vertex, where $1$ refers to the vertex being open and $0$ to the vertex being closed. Then we denote by $T_\infty^\omega = (\N_+, E(T_\infty^\omega))$ as the site-percolation on $T_\infty$
		\begin{align}\label{eq.percolation}
			E(T_\infty^\omega) := \{\{u,v\} \in E(T_\infty): \omega(u)=\omega(v) = 1\},
		\end{align}
		i.e. the edges only remain between open sites. We denote by $T_n$ and $T_n^\omega$ as their restriction on the vertex set $[n]: = \{1,2,\cdots,n\}$, which is a realisation of RRT and its site-percolation of size $n$.
		
		The site-percolation $T_\infty^\omega$ then gives the connected components called \textbf{clusters}. They form a partition $\Pi = (\Pi_i)_{i \in \N_+}$ of $\N_+$, i.e. 
		\begin{align}\label{eq.defBlock}
			\N_+ = \bigsqcup_{i=1}^\infty \Pi_i.
		\end{align} 
		Here $\bigsqcup$ stands for the disjoint union, and $\Pi_i$ is the the vertex set of a cluster which has the $i$-th smallest root. We also define $\Pi^{(n)}_{i} := \Pi_i \cap [n]$ as the restriction of $\Pi_i$ on $T_n^\omega$.
	\end{definition}
	We denote by $(\Omega, \mcl F, \P)$ the canonical coupling throughout the paper, and use $L^{q}$ as a shorthand for the Banach space $L^q(\Omega, \mcl F, \P)$ with $q \in [1, \infty]$.

	Our first result proves the convergence of the proportion of cluster sizes. Let $X_n := (X_n(0), X_n(1), X_n(2), X_n(3), \cdots)$ be the vector to stand for the numbers of clusters of different sizes in site-percolation, i.e. 
	\begin{align}\label{eq.defX}
		\forall k \geq 1, \quad X_n(k) := \#\Ll\{i \in \N_+: \vert \Pi^{(n)}_{i} \vert = k, \text{ and } \forall v \in \Pi^{(n)}_{i}, \omega(v)=1 \Rr\}.
	\end{align}
	Notice that $X_n(1)$ only counts the number of singletons of open vertices, so we define $X_n(0)$ as that of closed vertices
	\begin{equation}\label{eq.defX0}
		\begin{split}
			X_n(0) &:= \#\Ll\{i \in \N_+: \vert \Pi^{(n)}_{i} \vert = 1, \text{ and } \forall v \in \Pi^{(n)}_{i}, \omega(v)=0 \Rr\} \\
			&=\#\Ll\{v \in [n]: \omega(v) = 0\Rr\}.
		\end{split}
	\end{equation} 
	Moreover, $(X_n)_{n\in\N_+}$ is a Markov process, thanks to the canonical coupling in Definition \ref{def.Canonical}. Let $\e_k$ be the unit vector on the $k$-th coordinate, 
	in the canonical coupling space $(\Omega, \mcl F, \P)$, we have: 
	\begin{equation}\label{eq.Transition1}
		\begin{split}
			\P\Ll[X_{n+1} = X_n + \e_0 \vert  X_n\Rr] & = 1 - p, \\   
			\P\Ll[ X_{n+1} = X_n + \e_1 \vert X_n\Rr] & = \frac{p X_n(0)}{n}, \\    
			\P\Ll[ X_{n+1} = X_n - \e_k + \e_{k+1} \vert  X_n\Rr] & = \frac{p k  X_n(k)}{n}, \qquad k \geq 1.   
		\end{split}
	\end{equation}
	We can interpret each coordinate as one type of population, and treat the evolution as an infinite-type branching process.

	Given two vectors $\nu, f : \N \to \R$,  we define the following inner product 
	\begin{align}\label{eq.defTestf}
		\bracket{\nu, f} := \sum_{k = 0}^\infty \nu(k) f(k),
	\end{align} 
	and use it to measure the convergence of $\frac{X_n}{n}$ with a reasonable test function $f$. As discovered by \cite{kalay2014fragmentation}, a candidate for the limit proportion is the following: 
	\begin{equation}\label{eq.nu}
		\begin{split}
			\nu_p:\N \mapsto \R_+, \qquad \nu_p(0) = 1-p, \qquad &\nu_p(k) = c_{p} B\Ll(1 + \frac{1}{p}, k\Rr), \quad \forall k \in \N_+, \\
			&c_p = 1-p.
		\end{split}
	\end{equation}
	Here $B\Ll(x,y\Rr) = \frac{\Gamma(x)\Gamma(y)}{\Gamma(x+y)}$ is the standard Beta function, and $c_p$ is the constant of normalisation satisfying $c_{p}  \Ll(\sum_{k=1}^\infty B\Ll(1 + \frac{1}{p}, k\Rr) k\Rr) = p$. Note that normalising $(\nu_p(k))_{k\geq 1}$ to a probability measure results in the Yule--Simon distribution with parameter $\frac{1}{p}$. The Yule--Simon distribution was introduced in \cite{yule1925ii} by Yule in order to explain the power-law phenomena in biological
	genera, which was discovered in \cite{willis1922some} by him and Willis. Then this distribution was rediscovered several times in the history, including  in \cite{simon1955class} when Simon analysed the word frequency. Yule--Simon distribution also appears in various probability and statistics models, and see \cite{simkin2011re} for a summary. 
	\bigskip 
	
	\begin{theorem}[Law of large numbers to Yule--Simon distribution]\label{thm.main1} 
		Under the canonical coupling $(\Omega, \mcl F, \P)$, for every function $f:\N \mapsto \R$ satisfying ${\sup_{k \in \N} \frac{\vert f(k) \vert}{k} < \infty}$, 	and for every $q \in [1, \infty)$, we have
		\begin{align}
			\frac{1}{n}\bracket{X_n, f} \xrightarrow{n \to \infty}  \bracket{\nu_p, f}, \quad \P \text{-a.s. and in } L^q.
		\end{align}
	\end{theorem}

	We will give two proofs for this theorem. Our first proof transforms the site-percolation on RRT as a continuous-time infinite-type branching process. Instead of applying the general Crump--Mode--Jagers branching process theory \cite{iksanov2021asymptotic}, we find a more elementary approach with truncation, in which only the classical (finite) multi-type branching results and Perron--Frobenius theorem are needed. 
	
	Our second proof explains the link between the site-percolation and bond-percolation on RRT, as a result of the following coupling. For every edge $\{u,v\} \in E(T_\infty)$, if $u < v$, then we call $u$ \textbf{the parent} of $v$ and $v$ a \textbf{child} of $u$, and denote by $u = \mathbf{p}(v)$. By convention, we just set $\mathbf{p}(1) = 0$, so its parent is not on the recursive tree, and every other vertex $v \in \N_+ \setminus \{1\}$ has a parent in $\N_+$.
	
	\begin{figure}[t]
		\centering
		\includegraphics[width=0.8\textwidth]{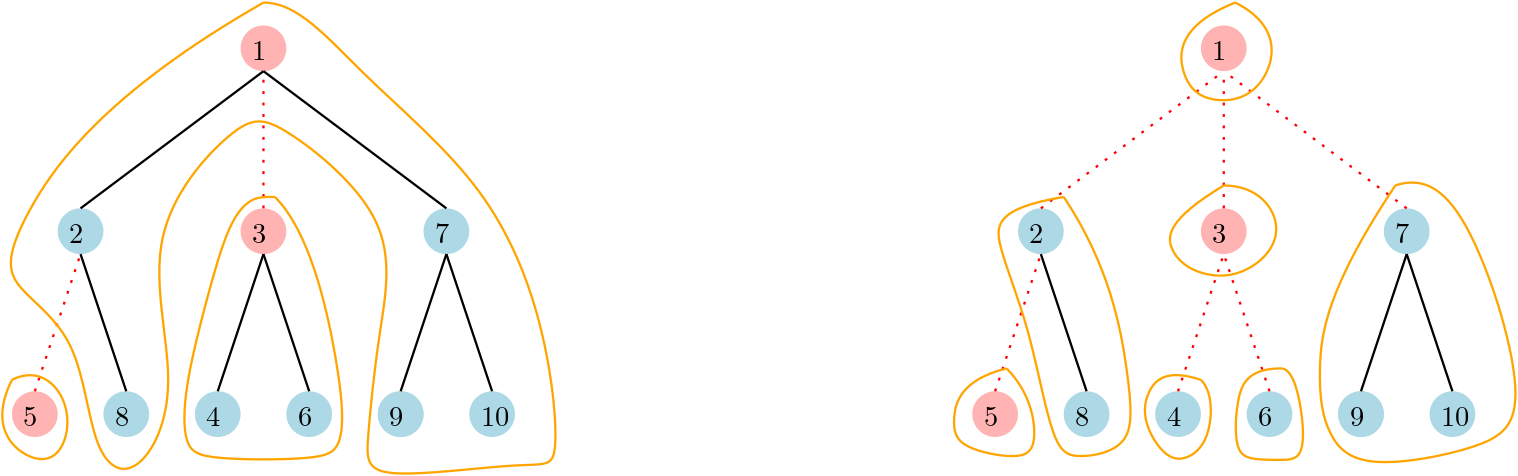}
		\caption{An illustration of the coupling between the bond-percolation (the figure on the left) and the site-percolation (the figure on the right). The vertices in red are closed, and those in blue are open.  The root-isolation is implemented, and the deleted edges are marked in red dotted line, while the orange circles stand for the clusters.}\label{fig.site-bond}
	\end{figure}
	\bigskip 
	\begin{lemma}[Site-bond-percolation coupling on RRT]\label{lem.site_bond_coupling}
		Given a random recursive tree $T_n$ with ${n \in \N_+ \cup \{\infty\}}$ and $\omega: \N_+ \to \{0,1\}$, we define  $\widetilde{T_n^{\omega}} = ([n], E(\widetilde{T_n^{\omega}}))$ as
		\begin{align}\label{eq.defBondPercolation}
			E(\widetilde{T_n^{\omega}}) := \{\{u,v\} \in [n] \times [n]: u = \mathbf{p}(v), \omega(v)=1 \},
		\end{align}
		i.e.\ every $\omega(v)$ determines whether $\{v, \mathbf{p}(v)\}$ is connected (for $v \geq 2$). Then $\widetilde{T_n^{\omega}}$ is a Bernoulli bond-percolation of parameter $p$ on RRT under $(\Omega, \mcl F, \P)$. 
		
		Moreover, we call the \textbf{root-isolation} for a cluster that we delete the edges between the root vertex and all its children vertices. Then we obtain the site-percolation $T_n^{\omega}$ from the bond-percolation $\widetilde{T_n^{\omega}}$ in the following way
		\begin{itemize}
			\item if $\omega(1) = 0$, implement the root-isolation for the cluster containing the vertex $1$ in $\widetilde{T_n^{\omega}}$;
			\item implement the root-isolation for all the other clusters.
		\end{itemize} 
	\end{lemma}
	
	Roughly speaking, it suffices to remove the edges directly attached to the root vertices in the bond-percolation to obtain the site-percolation (with special consideration for the cluster containing vertex 1); see Figure~\ref{fig.site-bond} for an illustration. Because the property of the bond-percolation $\widetilde{T_n^{\omega}}$ has been intensively discussed and a Yule--Simon distribution appears (see for example \cite{bertoin2022}), we can obtain similar results for the site-percolation using the root-isolation and the coupling between the two percolations. 
	The key is then reduced to studying the behaviour of the root-isolation operation in Lemma~\ref{lem.site_bond_coupling}, and this is related to the Ewens sampling formula.
	
	\medskip
	
	Theorem~\ref{thm.main1} studies the statistic of the clusters of a fixed size. We are then interested in the behaviour of the largest cluster. In the previous work \cite{kalay2014fragmentation}, the authors give a heuristic argument: as the proportion of the clusters of size $k$ is roughly $k^{-1-\frac{1}{p}}$, and the largest size, say $k_*$, should contain at least one cluster, then we have
	\begin{align}\label{eqn:heuristic}
		n \cdot \sum_{k = k_*}^\infty k^{-1-\frac{1}{p}} = 1 \Longrightarrow k_* \simeq n^p.
	\end{align}
	Our second result aims to give a rigorous proof. For every $q \in [1, \infty)$, we define the $\ell^q$ distance for $\mathbf{x} = (x_i)_{i \in \N_+} \in \R^{\N_+}$
	\begin{align}\label{eq.deflq}
		\norm{\mathbf{x}}_{\ell^{q}} := \Ll(\sum_{i=1}^\infty \vert x_i\vert^q\Rr)^{\frac{1}{q}}.
	\end{align}
	For $q = \infty$, then we define $\norm{\mathbf{x}}_{\ell^{\infty}} = \sup_{i \in \N_+} \vert x_i \vert$. Moreover, for every $\mathbf{x} \in \ell^q \cap {\R_+}^{\N_+}$, we can define its ranked version $\mathbf{x}^{\downarrow} = (x_i^{\downarrow})_{i \in \N_+}$ that
	\begin{align}\label{eq.defRankVersion}
		x_1^{\downarrow} \geq x_2^{\downarrow} \geq x_3^{\downarrow} \geq \cdots \geq 0.
	\end{align}
	Recall $(\Pi^{(n)}_i)_{i \in \N_+}$ the collection of vertex sets that partition $[n]$ in Definition~\ref{def.Canonical}, then \eqref{eq.defRankVersion} applies to $(\vert \Pi^{(n)}_i\vert)_{i \in \N_+}$ and we write its ranked version as
	\begin{align}\label{eq.defRank}
		\vert \Pi^{(n), \downarrow}_1 \vert \geq \vert \Pi^{(n), \downarrow}_2 \vert \geq 	\vert \Pi^{(n), \downarrow}_3 \vert \geq \cdots \geq 0.
	\end{align}
	Our second theorem states the convergence of this ranked cluster size sequence.
	

	\bigskip 
	\begin{theorem}[Scaling limit of the largest cluster sizes]\label{thm.main2}
		Under the canonical coupling probability space $(\Omega, \mcl F, \P)$, there exists a non-zero random variable $W = (W_i)_{i \in \N_+} \in \ell^{q}$ for all $q \in (\frac{1}{p}, \infty]$ such that  
		\begin{align}\label{eq.main2}
			n^{-p} \Ll(\vert \Pi^{(n), \downarrow}_1 \vert, \vert \Pi^{(n), \downarrow}_2 \vert, 	\vert \Pi^{(n), \downarrow}_3 \vert, \cdots\Rr) \xrightarrow[\ell^{q}]{n \to \infty} \Ll(W^{\downarrow}_1, W^{\downarrow}_2, W^{\downarrow}_3, \cdots\Rr) \qquad \P\text{-a.s.}  ,
		\end{align}
		where the limit is the ranked version of $W$.
	\end{theorem}
	\begin{remark}
		We also have the characterisation of the limit random variable $W$, see Proposition~\ref{prop.SiteClusterLimit} for details. Besides, our proof also applies to the bond-percolation setting to obtain a similar scaling limit, which completes \cite[Theorem~3.1, Lemma~3.3]{baur2015fragmentation}. 
	\end{remark}
	
	\medskip
	In the remaining of the paper, we will give the first proof of the Theorem~\ref{thm.main1} in Section~\ref{sec.branching} using branching processes. Then in Section~\ref{sec.bond}, we prove Lemma \ref{lem.site_bond_coupling}, explaining the link between the bond-percolation and site-percolation on RRT and use it to give a second proof of Theorem~\ref{thm.main1}. This link also helps us characterize the limit random variable in Theorem~\ref{thm.main2}, and we give the proof of its convergence in Section~\ref{sec.W}. Finally, Section~\ref{sec.discussions} includes some further discussions of this model.

	\section{Yule--Simon distribution 
		via branching processes}\label{sec.branching}
	In this section, we prove Theorem~\ref{thm.main1} from the viewpoint of branching processes.

	\subsection{Embedding an RRT into a branching process}
	We show at first that, as a Markov process (see \eqref{eq.Transition1}), $(X_n)_{n\in \N_+}$ can be considered being embedded into a branching process. Such idea was already used in \cite[Theorem~3.1]{baur2015fragmentation}. It is well-known that the infinite RRT can be realised by the continuous-time Yule process, that each vertex generates independently a new vertex with rate $1$ and obtains its label by the order of birth time. We generalised this idea to the site-percolation on RRT.
	
	\begin{definition}[Site-percolation on RRT via Yule process]\label{def.ContinuousFrag}
		Starting from a vertex with probability $p$ to be open and $(1-p)$ to be closed, every vertex generates independently a new vertex with exponential clock of rate $1$. The new vertex independently has probability $p$ to be open and $(1-p)$ to be closed, and is labelled by its birth order. We attach an edge between an open parent vertex and every of its open children vertices that it generated.  
	\end{definition}
	
	Note that the construction above is another way of defining site-percolation as in Definition \ref{def.Canonical}. The Yule process generates an infinite random recursive tree, and each vertex is determined whether open or closed when it is born.  To state it more precisely,  we introduce some more notations for Definition \ref{def.ContinuousFrag}. 
	For every vertex $v \in \N_+ \setminus \{1\}$, we denote by $\mathbf{b}(v)$ as its birth time and set $\mathbf{b}(1) = 0$ by convention. We denote by $\mathcal{T}_t = (V(\mathcal{T}_t), E(\mathcal{T}_t))$ as the random recursive tree formed by vertices appeared before $t$:
	\begin{align*}
		V(\mathcal{T}_t) &:= \{v \in \N_+: \mathbf{b}(v) \leq t\}, \\
		E(\mathcal{T}_t) &:= \{\{v, u\} \subset V(\mathcal{T}_t) \times V(\mathcal{T}_t): u = \mathbf{p}(v)\}.
	\end{align*}
	Thus, $V(\mathcal{T}_t)$ and $E(\mathcal{T}_t)$ are random sets. We use $\omega$, the same Bernoulli random variable in Definition~\ref{def.Canonical}, to indicate the state of every site, and denote by $\mathcal{T}_t^\omega = (V(\mathcal{T}_t), E(\mathcal{T}_t^\omega))$ the site-percolation on RRT at the time $t$
	\begin{align*}
		E(\mathcal{T}_t^\omega) := \{\{u,v\} \in E(\mathcal{T}_t): \omega(u) = \omega(v) = 1\}.
	\end{align*}
	The natural filtration for the site-percolation on RRT before time $t$ is defined by 
	\begin{align}\label{eq.defGt}
		\G_t := \sigma\Ll(V(\mathcal{T}_s)_{s \leq t}, E(\mathcal{T}_s)_{s \leq t}, (\omega(v))_{v \in V(\mathcal{T}_t)}\Rr),
	\end{align}
	then the canonical coupling space is now enriched as $(\Omega, \mcl F, (\G_t )_{t \geq 0}, \P)$. Especially, using $\vert V(\mathcal{T}_t) \vert$ for its total number of vertices at the moment $t$, then we define the stopping time 
	\begin{align}\label{eq.stopingTime}
		\tau_n := \inf\Ll\{t \in \R : \vert V(\mathcal{T}_t) \vert = n\Rr\}.
	\end{align}
	We can enrich the probability space defined in Definition~\ref{def.Canonical} as $(\Omega, \mcl F, (\G_t )_{t \geq 0}, \P)$ by identifying 
	\begin{equation}\label{eq.correspondence}
		(T_n^\omega)_{n \in \N_+} = (\mathcal{T}_{\tau_n}^\omega)_{n \in \N_+}.
	\end{equation}
	In another word,  $(T_n^\omega)_{n \in \N_+}$ is constructed using $(\mathcal{T}_t^\omega)_{t \geq 0}$ and embedded in the process of continuous time. 
	
	
	The equation \eqref{eq.correspondence} justifies a basic fact for the clusters in site-percolation.
	\begin{lemma}\label{lem.RRTiid}
		Given the size of the clusters in $T_n^\omega$, then each cluster follows the law of RRT independently.
	\end{lemma}
	\begin{proof}
		From Definition~\ref{def.ContinuousFrag} and using the thinning property of Poisson process, then every open vertex generates new open vertices independently with rate $p$. This gives us a family of independent Yule processes of density $p$, so each cluster in $(\mathcal{T}_{t}^\omega)_{t \geq 0}$ follows the law of RRT given its size. Viewing \eqref{eq.correspondence}, the clusters in $(T_n^\omega)_{n \in \N_+}$ satisfy the same property.
	\end{proof}

	Recall that $(X_n)_{n\in\N_+}$ counts the numbers of clusters in $(T_n^\omega)_{n \in \N_+}$. For convenience, we introduce similarly $Z_t := \Ll(Z_t(0), Z_t(1), Z_t(2), \ldots \Rr)$ as the counterpart for $(\mathcal{T}_{t}^\omega)_{t\geq 0}$: 
	\begin{align}\label{eq.defZ}
		Z_t(k) := \text{ number of open clusters of size } k \text{ in } \mathcal{T}_{t}^\omega, \qquad k \geq 1,
	\end{align}
	and
	\begin{align}\label{eq.defZ0}
		Z_t(0) := \text{ number of  closed vertices in } \mathcal{T}_{t}^\omega.
	\end{align} 
	Then thanks to \eqref{eq.correspondence}, we have 
	\begin{equation}\label{eqn:xz}
		(X_n)_{n\in\N_+}=(Z_{\tau_n})_{n\in\N_t}.
	\end{equation}
	Similar as \eqref{eq.Transition1}, $(Z_t)_{t \geq 0}$ can be seen as an infinite-type Markov branching process, with $\{0,1,2,\ldots\}$ for the types of individuals, and $Z_t(k)$ for the number of individuals of type $k$. Especially, the type $k \in \N_+$ stands for the open cluster of size $k$, and the type $0$ stands for the closed vertices.

	\begin{corollary}\label{cor.Z}
		In the space $(\Omega, \mcl F, (\G_t )_{t \geq 0}, \P)$, the process $(Z_t)_{t \geq 0}$ is a Markov branching process of the generator 	
		\begin{multline}\label{eq.Transition2}
			\L Z_t = \lim_{\Delta t \searrow 0} \frac{\E\Ll[Z_{t + \Delta t} - Z_t \, \vert \mcl G_t\Rr]}{\Delta t} \\
			= Z_t(0)\Big((1-p)\e_0 + p \e_1\Big)  + \sum_{k=1}^\infty Z_t(k)\Big(1-p)k\e_0 - p k \e_k + p k \e_{k+1} \Big).
		\end{multline}
		\begin{itemize}
			\item At $t=0,$ we start with one individual whose type is random and can only be of type $0$ with probability $(1-p)$ or of type $1$ with probability $p$.
			
			\item At any time $t > 0$, 
			\begin{itemize}
				\item an individual of type $0$ will produce a new individual of type $0$ at rate $(1-p)$, and produce a new individual of type $1$ at rate $p$;
				\item an individual of type $k \in \N_+$ will produce a new individual of type $0$ at rate $(1-p)k$, and change to type $(k+1)$ at rate $p k$.
			\end{itemize}
		\end{itemize}
	\end{corollary}

	The advantage of the continuous-time embedding in Definition~\ref{def.ContinuousFrag} is that, every individual evolves independently, and the transition rate only depends on the individual type. Then it suggests the following law of large numbers for branching processes. Recall that the limit measure $\nu_p$ is defined in \eqref{eq.nu}. 
	\begin{proposition}[Law of large numbers for $(Z_t)_{t \geq 0}$]\label{prop.LLN}
		There exists a positive random variable $\mathcal{W} \sim \operatorname{Exp}(1)$ following the standard exponential distribution, such that for every function $f:\N \mapsto \R$ satisfying $\sup_{n} \frac{\vert f(n) \vert}{n} < \infty$, we have 
		\begin{align}\label{eq.LLN}
			e^{-t}\bracket{Z_t, f} \xrightarrow{t \to \infty} \mathcal{W}\bracket{\nu_p, f}, \quad \P \text{-a.s. and in }  L^2.
		\end{align}
	\end{proposition}
	Notice that $(Z_t)_{t \geq 0}$ is of infinite type, so the classical result of multi-type branching process does not apply directly and requires some adaptation. We will prove it in the next subsection.
	
	\subsection{Truncated branching processes}
	
	In this part, we give an elementary proof of Proposition~\ref{prop.LLN} using the truncation technique. We define the truncated vector at level $h \in \N$ that 
	\begin{equation}\label{eq.Truncation}
		\begin{split}
			Z_t^{(h)} &:= \Ll(Z_t(0), Z_t(1), \cdots, Z_t(h), N^{\geq}_t(h+1), 0, 0, \cdots \Rr),\\
			N^{\geq}_t(h+1) &:= \sum_{k \geq h+1} k Z_t(k).    
		\end{split}
	\end{equation}
	Thus, $Z_t^{(h)}$ is the vector tracking the numbers of clusters until the level $h$, and ${N^{\geq}_t(h+1)}$ as the total number of open vertices in the clusters above this level; see Figure~\ref{fig.overweight} for an illustration.

	An important observation is that the evolution of $(Z_t^{(h)})_{t \geq 0}$ is autonomous, which reduces the infinite-type branching process to a (finite) multi-type branching process.
	\begin{lemma}
		Let $\G^{(h)}_t := \sigma((Z_s^{(h)})_{0 \leq s \leq t})$, then the truncated process $(Z_t^{(h)})_{t \geq 0}$ is a multi-type Markov branching process in $(\Omega, \mcl F, (\G^{(h)}_t)_{t \geq 0}, \P)$ with the generator 
		
		\begin{equation}\label{eq.TransitionTruncation}
			\begin{split}
				\L^{(h)}Z^{(h)}_t & =\lim_{\Delta t \searrow 0} \frac{\E\Ll[Z^{(h)}_{t + \Delta t} - Z^{(h)}_t \, \vert \mcl G^{(h)}_t\Rr]}{\Delta t} \\
				& = Z_t(0)\Ll((1-p)\e_0 + p \e_1\Rr)  + \sum_{k=1}^{h-1} Z_t(k)\Ll((1-p)k\e_0 - p k \e_k + p k \e_{k+1} \Rr) \\
				& \qquad + Z_t(h)\Ll((1-p)h\e_0 - p h \e_h + p h(h+1) \e_{h+1} \Rr)\\
				& \qquad + N^{\geq}_t(h+1)\Ll((1-p)\e_0 + p \e_{h+1}\Rr).    
			\end{split}
		\end{equation}
	\end{lemma}
	\begin{proof}
		One can check directly this generator by that of \eqref{eq.Transition2}. Here we give its probabilistic explanation. In the Markov branching process $(Z_t)_{t \geq 0}$, we treat the open clusters of size larger than $h$ as \textbf{over-weighted clusters}, and we use $N^{\geq}_t(h+1)$ to count the total number of the \textbf{over-weighted vertices} on them. Then, the transition law is stated as follows and illustrated in Figure~\ref{fig.overweight}:
		\begin{itemize}
			\item The transition law for the closed vertices, and the open clusters of size from $1$ to $(h-1)$ does not change.
			\item An open cluster of size $h$ has a rate $p h$ to become $(h+1)$ over-weighted vertices, and a rate $(1-p) h$ to produce a closed vertex.
			\item We do not have to consider the detailed sizes of the over-weighted clusters, but only count its production of vertices: every over-weighted vertex has rate $(1-p)$ to generate a closed vertex and $p$ to generate an over-weighted vertex.
		\end{itemize}
		This gives us the desired result. 
	\end{proof}
	
	
	\begin{figure}
		\centering
		\includegraphics[width=\linewidth]{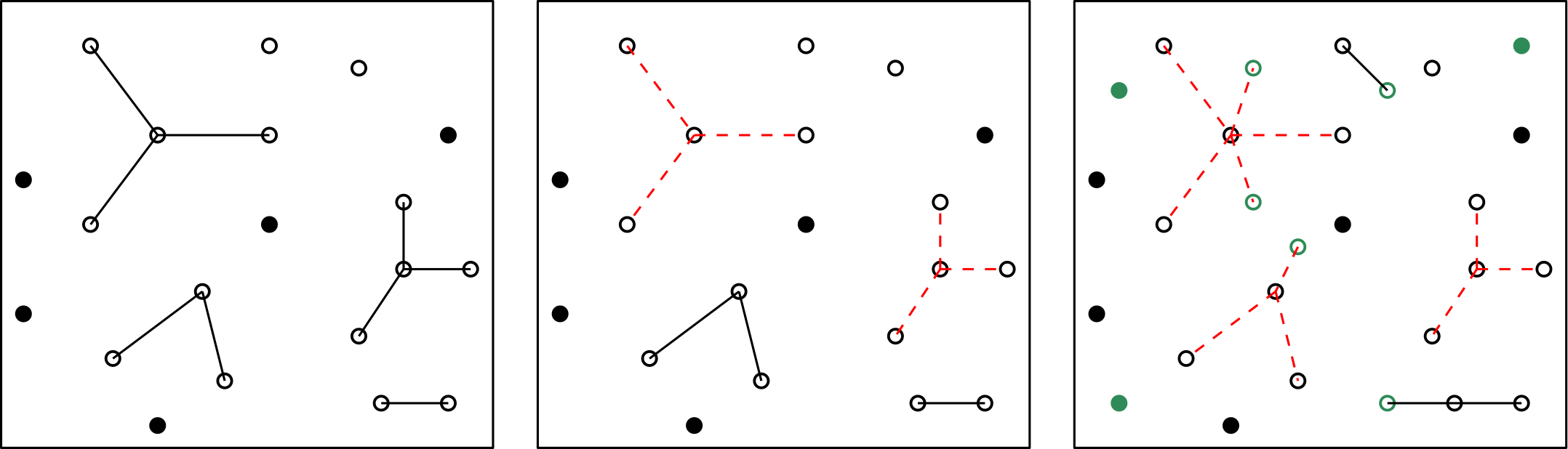}
		\caption{An illustration of the truncated process at the level $h = 3$. The sites in black are closed, while the transparent sites are open. The figure in the middle is the truncated version of the process in the figure on the left, where the clusters of size larger than $3$ become over-weighted and the dashed red line represents the deleted link there. In the figure on the right, more new vertices marked in green are attached following the dynamic of the truncated process. }\label{fig.overweight}
	\end{figure}

	We now make use of this truncated process to prove Proposition~\ref{prop.LLN}.
	\begin{proof}[Proof of Proposition~\ref{prop.LLN}]
		\textit{Step 1: law of large numbers for truncated process $Z_t^{(h)}$.}
		Since for every $h \in \N$, the truncated process $(Z_t^{(h)})_{t \geq 0}$ is a multi-type Markov branching process, the classical law of large numbers theory using Perron--Frobenius theorem applies; see the reference \cite[Theorem~1]{athreya1968some}. This suggests that for every test function $f$
		\begin{align}\label{eq.TruncateLLN}
			e^{- \lambda^{(h)} t} \bracket{Z_t^{(h)}, f} \xrightarrow{t \to \infty} \mathcal{W}^{(h)} \bracket{\nu_{p}^{(h)}, f}, \qquad \P\text{-a.s. and in } L^2,
		\end{align}
		where $\lambda^{(h)} > 0$ is the Malthusian exponent of the generator $\lambda^{(h)}$ defined in \eqref{eq.TransitionTruncation},  and $\nu_{p}^{(h)}$  is the associated eigenvector such that 
		\begin{align}\label{eq.eigenvectorTruncated}
			\L^{(h)} \nu_{p}^{(h)} = \lambda^{(h)} \nu_{p}^{(h)}. 
		\end{align}
		We also make a normalisation such that $\nu_{p}^{(h)}(0) = 1-p$, then $\nu_{p}^{(h)}$ is not necessarily a vector of measure, but the first coordinate does not depend on $h$. 
		Here $\mathcal{W}^{(h)}$ is a positive random variable. Because the vector $Z_t^{(h)}$ and limit $\nu_{p}^{(h)}$ are of compact support, we do not have to give specific conditions for the test function $f$ in \eqref{eq.TruncateLLN}. Moreover, we remark that the classical result in \cite[Theorem~1]{athreya1968some} implies the almost sure convergence, and the $L^2$ convergence comes from a uniform control of the moment. 
		
		\smallskip
		
		\textit{Step 2: finite dimensional convergence in $Z_t$.}
		We need to identify the Malthusian exponent $\lambda^{(h)}$, the limit random variable $\mathcal{W}^{(h)}$, and the left eigenvector $\nu_{p}^{(h)}$. We denote by $\x{} : \N \mapsto \R$ linear function that $\x{} (n) = n$, then it gives us an important functional that 
		\begin{align}
			N_t := \bracket{Z_t, \x{} \vee 1} = Z_t(0) + \sum_{k=1}^\infty k Z_t(k),
		\end{align}
		which is the total number of vertices at the moment $t$. In Definition~\ref{def.ContinuousFrag}, regardless of the type of vertices, every vertex produces a new vertex in exponential time and $(N_t)_{t \geq 0}$ is a Yule process. Then the classical limit result of Yule process (see \cite[Remark 5, Page 130]{AN2004}) applies that 
		\begin{align}\label{eq.TotalLLN}
			e^{-  t} \bracket{Z_t, \x{} \vee 1} =  e^{-  t} N_t \xrightarrow{t \to \infty} \mathcal{W}, \qquad \P\text{-a.s. and in } L^2,
		\end{align}
		where $\mathcal{W}$ is an exponential random variable. We introduce another test function $g^{(h)} : \N \mapsto \R$ as 
		\begin{align*}
			g^{(h)}(0) = 1, \qquad g^{(h)}(k) = k,  \qquad\forall 1 \leq k \leq h, \\
			g^{(h)}(h+1) = 1, \qquad g^{(h)}(j) = 0, \qquad \forall j > h+1,  
		\end{align*}
		Then $N_t$ can also be written as a functional of $Z_t^{(h)}$: 
		\begin{align}\label{eq.IdNt}
			N_t = Z_t(0) + \sum_{k=1}^\infty k Z_t(k) \stackrel{\eqref{eq.TransitionTruncation}}{=} Z_t(0) + \sum_{k=1}^h k Z_t(k) + N^{\geq}_t(h+1) = \bracket{Z_t^{(h)}, g^{(h)}}.
		\end{align}
		This identity together with \eqref{eq.TruncateLLN} suggests that $N_t$ has a Malthusian exponent $\lambda^{(h)}$. Comparing \eqref{eq.IdNt} with \eqref{eq.TotalLLN}, we identify that $\lambda^{(h)} = 1$.

		We then turn to the eigenvector $\nu_{p}^{(h)}$. Using the characterisation of the eigenvector \eqref{eq.eigenvectorTruncated} and $\lambda^{(h)} = 1$, we have 
		\begin{align}\label{eq.linearSystem}
			\Ll\{\begin{array}{lll}
				\nu_{p}^{(h)}(0) &= (1-p)\Ll(\nu_{p}^{(h)}(0) + \sum_{k=1}^{h} k \nu_{p}^{(h)}(k) + \nu_{p}^{(h)}(h+1)\Rr), &\\
				\nu_{p}^{(h)}(1) &= p \Ll(\nu_{p}^{(h)}(0) - \nu_{p}^{(h)}(1)\Rr), &\\
				\nu_{p}^{(h)}(k) &= p \Ll((k-1)\nu_{p}^{(h)}(k-1) - k \nu_{p}^{(h)}(k)\Rr), &\forall 1 \leq k \leq h, \\
				\nu_{p}^{(h)}(h+1) &= p \Ll(h(h+1) \nu_{p}^{(h)}(h) + \nu_{p}^{(h)}(h+1)\Rr), &\\
				\nu_{p}^{(h)}(j) &= 0, &\forall j > h+1.
			\end{array}\Rr.
		\end{align} 
		This liner system gives an iteration
		\begin{align}\label{eq.NuIteration1}
			\nu_{p}^{(h)}(1) = \Ll(\frac{1}{1+\frac{1}{p}}\Rr) \nu_{p}^{(h)}(0), \qquad \nu_{p}^{(h)}(k+1) = \Ll(\frac{k}{k+1+\frac{1}{p}}\Rr) \nu_{p}^{(h)}(k), \quad 1 \leq k \leq h-1. 
		\end{align}
		As we fix $\nu_{p}^{(h)}(0) = 1-p$, the iteration \eqref{eq.NuIteration1} gives us the result $\nu_{p}^{(h)}(k) = c_{p} B\Ll(1 + \frac{1}{p}, k\Rr)$ for all $1 \leq k \leq h$. Then the following limit exists
		\begin{align}\label{eq.NuIteration2}
			\nu_{p}(k) := \lim_{h \to \infty} \nu_{p}^{(h)}(k),
		\end{align}
		because the value of $\nu_{p}^{(h)}(k)$ is constant from $h \geq k$.
		
		Finally, we identify the random variable $\mathcal{W}^{(h)}$. We make use of this exponent and test \eqref{eq.TruncateLLN} with $\Ind{\cdot = 0}$,  
		\begin{align*}
			\mathcal{W}^{(h)} \nu_{p}^{(h)}(0) = \lim_{t \to \infty} e^{- \lambda^{(h)} t} \bracket{Z_t^{(h)}, \Ind{\cdot = 0}} = \lim_{t \to \infty} e^{-  t} Z_t(0).
		\end{align*}
		Recall that we fix $\nu_{p}^{(h)}(0) = (1-p)$ for all $h \in \N$, thus $\mathcal{W}^{(h)}$  does not depend on $h$, 
		and we denote by $\mathcal{W}^{(h)} = \mathcal{W}'$. 		We pick a very special case that $h=0$, where $Z_t^{(0)}(0)$ and $Z_t^{(0)}(1)$ count respectively the closed vertices and the open vertices, so we have $\nu_{p}^{(0)}(0)= 1-p$ and $\nu_{p}^{(0)}(1) = p$ from \eqref{eq.linearSystem}. Then we test \eqref{eq.TruncateLLN} with $\Ind{\cdot = 1}$ and obtain 
		\begin{align*}
			\lim_{t \to \infty} e^{-t} Z^{(0)}_t(1) =  \mathcal{W}' \nu_{p}^{(0)}(1) = \mathcal{W}'p.
		\end{align*}
		We combine the identity \eqref{eq.IdNt} and \eqref{eq.TotalLLN}
		\begin{align}\label{eq.WW}
			\mathcal{W} = \lim_{t \to \infty}   e^{-  t} N_t = \lim_{t \to \infty}  e^{-t} \bracket{Z_t^{(0)}, g^{(0)}} = \mathcal{W}'(\nu_{p}^{(0)}(0) + \nu_{p}^{(0)}(1)) = \mathcal{W}',
		\end{align}
		which implies $\mathcal{W}'=\mathcal{W}$.
		
		Therefore, we have identified all the quantities in \eqref{eq.TotalLLN}. We now test it with the indicator function $\Ind{\cdot = k}$ and obtain	
		\begin{align}\label{eq.LLNFinite}
			\lim_{t \to \infty} e^{-t} Z_t(k) = \lim_{h \to \infty} \lim_{t \to \infty} e^{-t} Z_t^{(h)}(k) = \lim_{h \to \infty} \mathcal{W}' \nu_{p}^{(h)}(k) = \mathcal{W}\nu_{p}(k).
		\end{align}

		\smallskip
		
		\textit{Step 3: convergence for general test functions.} In this part, we need to extend the result to a general test function. The idea is similar to the dominated convergence theorem, that the point-wise convergence and the convergence of a dominated functional will result in our desired result. 
		
		For a general function $f$ with ${\sup_{n \in \N_+} \frac{\vert f(n) \vert}{n} < \infty}$, we make a decomposition that 
		\begin{align*}
			f = f^{(h)}_\leq + f^{(h)}_>, \qquad f^{(h)}_\leq (x) := f(x) \Ind{x \leq h}, \qquad f^{(h)}_> (x) := f(x) \Ind{x > h}.
		\end{align*}
		Then we can make use of truncation that
		\begin{equation}\label{eq.General}
			\begin{split}
				&\Ll\vert e^{-t} \bracket{Z_t, f} -  \mathcal{W} \bracket{\nu_p, f} \Rr\vert\\
				&\leq \Ll\vert e^{-t} \bracket{Z_t, f^{(h)}_\leq} -   \mathcal{W} \bracket{\nu_p, f^{(h)}_\leq} \Rr\vert + \Ll\vert e^{-t} \bracket{Z_t, f^{(h)}_>} -   \mathcal{W} \bracket{\pi_p, f^{(h)}_>} \Rr\vert\\
				&\leq \Ll\vert e^{-t} \bracket{Z_t, f^{(h)}_\leq} -   \mathcal{W} \bracket{\nu_p, f^{(h)}_\leq} \Rr\vert + \Ll(\sup_{n} \frac{\vert f(n)\vert }{n}\Rr)\Big(\Ll\vert e^{-t} \bracket{Z_t, (\x{} \vee 1) \Ind{\cdot > h}}\Rr\vert + \Ll\vert   \mathcal{W} \bracket{ \nu_p, (\x{} \vee 1) \Ind{\cdot > h}}\Rr\vert\Big)\\
			\end{split}    
		\end{equation}
		We focus on the almost surely convergence and the $L^2$-convergence can be done similarly.
		
		For its first term, we have
		\begin{align*}
			\Ll\vert e^{-t} \bracket{Z^{(h)}_t, f^{(h)}_\leq} - \mathcal{W} \bracket{\nu^{(h)}_p, f^{(h)}_\leq} \Rr\vert \xrightarrow[a.s.]{t \to \infty} 0,
		\end{align*}
		by the finite dimensional convergence from \eqref{eq.TruncateLLN} and \eqref{eq.LLNFinite}. 
		
		For its second term, since \eqref{eq.TotalLLN}  and \eqref{eq.TruncateLLN} give us
		\begin{align*}
			e^{-t} \bracket{Z_t, (\x{} \vee 1) \Ind{\cdot \leq h}} &\xrightarrow[a.s.]{t \to \infty}   \mathcal{W}  \bracket{ \nu_p,  (\x{} \vee 1) \Ind{\cdot \leq h}}, \\
			e^{-t} \bracket{Z_t, (\x{} \vee 1)} &\xrightarrow[a.s.]{t \to \infty}   \mathcal{W}  \bracket{ \nu_p, (\x{} \vee 1)}, 
		\end{align*}
		we make the difference and obtain
		\begin{align*}
			e^{-t} \bracket{Z_t, (\x{} \vee 1) \Ind{\cdot > h}} \xrightarrow[a.s.]{t \to \infty}   \mathcal{W}  \bracket{ \nu_p, (\x{} \vee 1)\Ind{\cdot > h}}.
		\end{align*}
		We put these estimates back to \eqref{eq.General} and obtain 
		\begin{align*}
			\limsup_{t \to \infty}\Ll\vert e^{-t} \bracket{Z_t, f} -   \mathcal{W} \bracket{\nu_p, f} \Rr\vert \leq 2\Ll(\sup_{n} \frac{\vert f(n)\vert }{n}\Rr)   \mathcal{W}  \bracket{ \nu_p, (\x{} \vee 1)\Ind{\cdot > h}}, \qquad \P\text{-a.s.} .
		\end{align*}
		Then we let $h \to \infty$ and conclude the convergence for general test function.
		
	\end{proof}
	
	\subsection{Proof of Theorem~\ref{thm.main1}} 
	We now use Proposition~\ref{prop.LLN} to prove the main theorem.
	\begin{proof}[Proof of Theorem~\ref{thm.main1}]
		Recall that \eqref{eqn:xz} entails
		\begin{align*}
			\bracket{X_n, f} = \bracket{Z_{\tau_n}, f}.
		\end{align*} 
		Using the definition of $\tau_n$ in \eqref{eq.stopingTime} and the identity \eqref{eq.TotalLLN}, we also have 
		\begin{align*}
			n = \bracket{X_n, [x]\vee 1} = \bracket{Z_{\tau_n}, [x]\vee 1}.
		\end{align*}
		Therefore, concerning the almost sure convergence, Proposition~\ref{prop.LLN} applies
		\begin{align}\label{eq.main1as}
			\frac{\bracket{X_n, f}}{n} = \frac{e^{- \tau_n} \bracket{Z_{\tau_n}, f}}{e^{- \tau_n}  \bracket{Z_{\tau_n}, [x]\vee 1}} \xrightarrow{n \to \infty} \frac{\mathcal{W} \bracket{\nu_p, f}}{\mathcal{W} \bracket{\nu_p, [x]\vee 1}} = \bracket{\nu_p, f}, \qquad \P\text{-a.s.} .
		\end{align} 
		Here we also use the normalisation that $\bracket{\nu_p, [x]\vee 1} = 1$ in the statement of Theorem~\ref{thm.main1}.
		
		Concerning the $L^q$ convergence, we notice that 
		\begin{align*}
			\vert \bracket{X_n, f} \vert \leq \Ll(\sup_{k \in \N} \frac{\vert f(k) \vert}{k}\Rr) \bracket{X_n, [x]\vee 1} = \Ll(\sup_{k \in \N} \frac{\vert f(k) \vert}{k}\Rr) n.
		\end{align*}
		Here we apply the identity \eqref{eq.IdNt} and the definition \eqref{eq.stopingTime} once again. This implies that 
		\begin{align*}
			\Ll\vert \frac{\bracket{X_n, f}}{n} \Rr\vert \leq \Ll(\sup_{k \in \N} \frac{\vert f(k) \vert}{k}\Rr).
		\end{align*}
		Viewing this bound and \eqref{eq.main1as}, we can apply the dominated convergence theorem to deduce the $L^q$ convergence. 
	\end{proof}
	
	\section{Yule--Simon distribution via  
		bond-percolation on RRT}\label{sec.bond}
	We give another proof of Theorem~\ref{thm.main1} using the coupling between the site-percolation and the bond-percolation on RRTs.
	We start from the justification of Lemma~\ref{lem.site_bond_coupling}. 
	
	\begin{proof}[Proof of Lemma~\ref{lem.site_bond_coupling}]
		Note that a vertex is either closed or open in $T^\omega_n$ as determined by $\omega$. In the bond-percolation $\widetilde{T^\omega_n}$, we just remove the link between the closed vertex and its parent. To obtain the site-percolation defined in Definition~\ref{def.Canonical}, we need to further remove the link between the closed vertices and their children vertices. Notice that every closed vertex is the root of its associated cluster in $\widetilde{T^\omega_n}$, then the operation just follows the description in  Lemma~\ref{lem.site_bond_coupling}. Moreover, Every bond in $\widetilde{T^\omega_n}$ is associated to one unique child vertex, so it has Bernoulli probability $p$ to remain independently under $(\Omega, \mcl F, \P)$. This justifies that the $\widetilde{T^\omega_n}$ has a law of Bernoulli percolation on RRT.
	\end{proof}

	The second proof of Theorem~\ref{thm.main1} relies on the following two observations. The first one is a similar result of Theorem~\ref{thm.main1} for the site-percolation: for every $k \in \N_+$, we denote by 	
	\begin{align}\label{eq.defY}
		Y_n(k) := \#\{\text{the open clusters of size } k \text{ in } \widetilde{T_n^\omega}\}.
	\end{align}
	Then for $\tilde{c}_p= \frac{1}{p} - 1$, we have 
	\begin{align}\label{eq.YuleSimonY}
		\frac{Y_n(k)}{n} \xrightarrow{n \to \infty} \tilde{c}_p B\left(1+\frac{1}{p}, k\right), \qquad \forall k \in \N_+, \qquad \P\text{-a.s.} .
	\end{align}
	This was obtained in Simon's model when he studied the frequency of words in \cite{simon1955class}; For a more recent reference, see its proof and related discussion in \cite[Section~5.3]{bertoin2022}.

	The second observation concerns the coupling in Lemma~\ref{lem.site_bond_coupling} and we are interested in the size distribution of the subtrees when removing the root of an RRT. This is related to the celebrated Ewens sampling formula (\cite[(1.3)]{arratia2003logarithmic}): given an RRT of size $(k+1)$ and we remove its root, and denote by $C_k(j)$ the number of clusters of size $j$. Then we have 
	\begin{align}\label{eq.Ewens}
		\P\Ll[C_k(1)=a_1, C_k(2)=a_2,\cdots, C_k(k)=a_{k}\Rr]={\bf 1}_{\sum_{j=1}^{k}j a_j=k}\prod_{j=1}^{k}\frac{1}{j^{a_j}a_j!}.
	\end{align}
	Moreover, we can calculate its mean and second moment (see \cite[Lemma 1.1]{arratia2003logarithmic}) 
	\begin{align}\label{eq.EwensMean}
		\E[C_k(j)]=\frac{1}{j},\qquad \E\Ll[C_k(j)^2\Rr]=\frac{1}{j} + \frac{1}{j^2}, \qquad  \forall 1\leq j\leq k.
	\end{align}
	
	\begin{proof}[Second proof of Theorem~\ref{thm.main1}]
		Now we apply these two observations in our model. 
		
		\textit{Step~1: site-bond-percolation coupling.} We implement the root-isolation defined in Lemma~\ref{lem.site_bond_coupling} for all the clusters in the bond-percolation $\widetilde{T^\omega_n}$ to yield the site-percolation (with special consideration to the cluster containing vertex 1). Following the definition of $Y_n(k+1)$ in \eqref{eq.defY}, for every $k \in \N_+$, we give an arbitrary order to all the clusters of size $(k+1)$ in the bond-percolation $\widetilde{T^\omega_n}$. Then we  denote by $C_k^{(i)} = \Ll(C_k^{(i)}(j) \Rr)_{1 \leq j \leq k}$ as the size vector appearing in \eqref{eq.Ewens} for the $i$-th cluster of size $(k+1)$. We claim that 
		\begin{equation}\label{eq.XYLimit}
			\begin{split}
				\lim_{n \to \infty} \frac{X_n(j)}{n} &= \lim_{n \to \infty} \frac{1}{n}\sum_{k=j}^\infty \sum_{i=1}^{Y_n(k+1)} C^{(i)}_{k}(j), \qquad \forall j \in \N_+, \\
				\lim_{n \to \infty} \frac{X_n(0)}{n} &= \lim_{n \to \infty} \frac{\sum_{k=1}^\infty Y_n(k)}{n}.
			\end{split}
		\end{equation}
		Roughly speaking, the right-hand side of the first equation just counts the number of clusters  of size $j$ after isolating the root vertices in the bond-percolation.   Recall that to obtain site-percolation from bond-percolation, we only remove the link between the vertex $1$ and its children vertices when vertex 1 is closed (i.e.\ $\omega(1)=0$), see Lemma~\ref{lem.site_bond_coupling}. Then the first equation holds even without the limit operator if the vertex $1$ is closed. 
		In the case that vertex $1$ is open (i.e.\ $\omega(1)=1$), as the typical size of the cluster containing the vertex $1$ is of order $n^p$ (see \cite[Theorem~3.1.(i)]{baur2015fragmentation}) and negligible compared to $n$, the first equation holds with the limit operator. 
		
		It remains to prove the convergence of the right-hand side in \eqref{eq.XYLimit}. We just focus on the almost sure convergence for one coordinate $k \in \N_+$ as in Step~2 of the proof of Proposition~\ref{prop.LLN}, because the other results can be deduced similarly. Inspired by the mean-variance decomposition (see \eqref{eq.EwensMean})
		\begin{align*}
			C^{(i)}_{k}(j) = \underbrace{\frac{1}{j}}_{\text{``mean"}} + \underbrace{\Ll(C^{(i)}_{k}(j) - \frac{1}{j}\Rr)}_{\text{``variance"}},
		\end{align*}
		we decompose the equation into two terms 
		\begin{align*}
			\frac{1}{n}\sum_{k=j}^\infty \sum_{i=1}^{Y_n(k+1)} C^{(i)}_{k}(j) &= \frac{1}{n} \sum_{k=j}^\infty \sum_{i=1}^{Y_n(k+1)} \frac{1}{j} + R_n(j)\\
			R_n(j) &:=  \frac{1}{n} \sum_{k=j}^\infty \sum_{i=1}^{Y_n(k+1)} \Ll(C^{(i)}_{k}(j) - \frac{1}{j}\Rr),
		\end{align*}
		where $R_n(j)$ is the remainder. We treat the convergence of the two terms separately.
		
		\smallskip
		
		\textit{Step~2: convergence of mean.} For the first term, the heuristic argument is to put the limit from \eqref{eq.YuleSimonY} that 
		\begin{align}\label{eq.XYheuristic}
			\lim_{n \to \infty}  \sum_{k=j}^\infty  \frac{Y_n(k+1)}{n} \frac{1}{j} &=  \tilde{c}_p\sum_{k= j}^{\infty}B\left(1+\frac{1}{p}, k+1\right)\frac{1}{j} \qquad \P\text{-a.s.}.
		\end{align}
		The exchange between the limit and the sum requires some more careful justification, which can be illustrated by the following argument: because the total number of vertices is $n$, we have $\sum_{k=1}^\infty kY_n(k) = n$, and the Markov inequality implies
		\begin{align}\label{eq.YTrivialBound}
			\sum_{k=M+1}^\infty Y_n(k+1) \leq \frac{n}{M}.
		\end{align}
		Then we obtain
		\begin{align*}
			\lim_{n \to \infty}  \sum_{k=j}^M  \frac{Y_n(k+1)}{n} \frac{1}{j} \leq  \lim_{n \to \infty}  \sum_{k=j}^\infty  \frac{Y_n(k+1)}{n} \frac{1}{j} \leq \lim_{n \to \infty}  \sum_{k=j}^M  \frac{Y_n(k+1)}{n} \frac{1}{j} + \frac{1}{M j}.
		\end{align*}
		We can send $n$ to infinity at first, and then let $M$ be arbitrarily large to establish \eqref{eq.XYheuristic}. 
		
		We calculate the value for the right-hand side of \eqref{eq.XYheuristic}
		\begin{align}
			\begin{split}\label{eqn.whyyulesimon}			\tilde{c}_p\sum_{k= j}^{\infty}B\left(1+\frac{1}{p}, k+1\right)\frac{1}{j} &= \frac{1}{j} \tilde{c}_p\sum_{k= j}^{\infty} \int_0^1 x^{\frac{1}{p}} (1-x)^{k} \, \d x \\
				&= \frac{1}{j} \tilde{c}_p \int_0^1 x^{\frac{1}{p}} \Ll(\sum_{k= j}^{\infty}(1-x)^{k}\Rr) \, \d x \\
				&= \frac{1}{j} \tilde{c}_p \int_0^1 x^{\frac{1}{p}-1} (1-x)^{j} \, \d x\\
				&= \frac{1}{j} \tilde{c}_p B\left(\frac{1}{p}, j+1\right) = c_p B\left(1+\frac{1}{p}, j\right).
			\end{split}
		\end{align}
		Here from the first line to the second line, we use the monotone convergence theorem. From the third line to the fourth line, we just use the definition of the Beta function. The equality in the fourth line comes from the identity that $\Gamma(\frac{1}{p})\Gamma(j+1) = pj\Gamma(1+\frac{1}{p})\Gamma(j)$. Recall that $c_p = 1-p$ defined in \eqref{eq.nu}, and $\tilde{c}_p = 1/p - 1$ defined around \eqref{eq.YuleSimonY}. 
		
		\smallskip
		
		\textit{Step~3: convergence of the fluctuation.}
		We study the second moment of the remainder term $R_n(j)$
		\begin{align*}
			\E[(R_n(j))^2] &= \E\Ll[\Ll(\frac{1}{n} \sum_{k=j}^\infty \sum_{i=1}^{Y_n(k+1)} \Ll(C^{(i)}_{k}(j) - \frac{1}{j}\Rr)\Rr)^2\Rr]\\
			&=  \E\Ll[\E\Ll[\Ll(\frac{1}{n} \sum_{k=j}^\infty \sum_{i=1}^{Y_n(k+1)} \Ll(C^{(i)}_{k}(j) - \frac{1}{j}\Rr)\Rr)^2 \Big \vert (Y_n(k+1))_{k \in \N_+}\Rr]\Rr]\\
			&= \E\Ll[ \sum_{k=j}^\infty \frac{Y_n(k)}{n^2} \var[C_{k}(j)]\Rr]\\
			&=  \E\Ll[ \sum_{k=j}^\infty \frac{Y_n(k)}{n^2} \frac{1}{j}\Rr]\\
			&\leq \frac{1}{nj^2}.
		\end{align*}
		Here the argument for the first line to the third line is the following: by Lemma~\ref{lem.RRTiid}, conditioned on the size, every cluster in the site-percolation also follows the law of RRT.  Therefore, conditioned on $Y_n(k+1)$,
		the variables $\{C^{(i)}_{k}(j)\}_{1 \leq i \leq Y_n(k+1)}$ are i.i.d., then \eqref{eq.EwensMean} applies. 
		From the fourth line to the fifth line, we use the bound $\sum_{k=j}^\infty Y_n(k)$ as in \eqref{eq.YTrivialBound}. This implies that ${R_n(j) \xrightarrow[n \to \infty]{\P} 0}$ and ${R_{N^2}(j) \xrightarrow[N \to \infty]{a.s.} 0}$ for the subsequence $(N^2)_{N\in\N_+}$ by Borel--Cantelli lemma. 
		
		All the analysis above results in $\frac{ X_{N^2}(j)}{N^2} \xrightarrow[N \to \infty]{a.s.} c_p B\left(1+\frac{1}{p}, j\right)$. Then for general ${n \in [N^2, (N+1)^2)}$, we observe that $\vert  X_{N^2}(j) -  X_{n}(j)\vert \leq 2 (n - N^2) \leq 4 (\sqrt{n}+1)$ from \eqref{eq.Transition1}. Therefore, we have 
		\begin{align*}
			\lim_{n \to \infty} \frac{ X_{n}(j)}{n} = 	\lim_{n \to \infty} \left(\frac{ X_{n}(j) -  X_{N^2}(j)}{n} + \frac{ X_{N^2}(j)}{N^2} \frac{N^2}{n}\right) =   c_p B\left(1+\frac{1}{p}, j\right), \quad \P\text{-a.s.} .
		\end{align*}
		This gives us the desired result. 
	\end{proof}
	\begin{remark}
		For the bond-percolation of a large RRT, a ``typical" (randomly chosen) cluster has size following asymptotically the Yule--Simon distribution with parameter $\frac{1}{p}$ (see \eqref{eq.YuleSimonY}). The main idea of this proof is to use this fact and the site-bond percolation coupling to deduce the number of clusters of a fixed size after the site-percolation. The involves finding the expectation of the number of clusters of a fixed size after removing the root of a typical cluster and applying the strong law of large numbers. The Yule--Simon distribution arises in the site-percolation due to the computations in \eqref{eqn.whyyulesimon}.
	\end{remark}

	\section{Scaling limit of the largest cluster sizes}\label{sec.W}
	\subsection{Characterisation of the limit random variables}
	In this subsection, we will prove firstly a weaker version of Theorem~\ref{thm.main2}, which is the convergence of $n^{-p} \vert \Pi^{(n)}_i \vert $ for any given $i$. Our proof relies on the similar result of the bond-percolation on RRT and the coupling in Lemma~\ref{lem.site_bond_coupling}. Let $\widetilde{\Pi} = (\widetilde{\Pi}_i)_{i \in \N_+}$ be a partition of $\N_+$ by the clusters from the bond-percolation $\widetilde{T_\infty^{\omega}}$ defined in Lemma~\ref{lem.site_bond_coupling}, where $\widetilde{\Pi}_i$ contains the $i$-th smallest root vertex, and let $\widetilde{\Pi}^{(n)}_i := \widetilde{\Pi}_i \cap [n]$ be the restriction on $\widetilde{T_n^{\omega}}$. The following result states the convergence of $n^{-p} \vert \widetilde{\Pi}^{(n)}_i \vert$, which was proved in \cite[Theorem~3.1.(i), Lemma~3.3, Corollary]{baur2015fragmentation} and can also be found explicitly in \cite[Proposition~2.15]{guerin2023fixed}.
	
	\begin{proposition}{\cite{baur2015fragmentation,guerin2023fixed}}\label{prop.BondClusterLimit}
		Under the canonical coupling space $(\Omega, \mcl F, \P)$, there exists a random variable $\widetilde{W} = (\widetilde{W}_i)_{i \in \N_+} \in \ell^{q}$ such that 
		\begin{align}\label{eq.BondClusterLimit}
			n^{-p} \vert \widetilde{\Pi}^{(n)}_i \vert \xrightarrow[\text{a.s.}]{n \to \infty} \widetilde{W}_i, \qquad \forall i \in \N_+.
		\end{align}
		Moreover, the random variable $\widetilde{W}$ has the following characterisation:
		\begin{itemize}
			\item $\widetilde{W}_1$ is a Mittag--Leffler random variable with parameter $p$.
			\item For every $i \geq 2$, the variable $\widetilde{W}_i \stackrel{d}{=} \widetilde{W}_{1,i} (\beta_{\sigma_i})^p$, where
			\begin{itemize}
				\item $\widetilde{W}_{1,i}$ is a Mittag--Leffler random variable with parameter $p$;
				\item $(\beta_k)_{k \in \N_+}$ are independent Beta random variables with parameters $(1, k-1)$;
				\item $\sigma_{j} - 1$ are dependent negative binomial random variables with parameters $(j-1, 1-a)$, such that $\sigma_j = \sigma_{j-1} + G_j$ and $G_j$ ($\geq 1$) has geometric distribution of parameter $(1-p)$
				and is independent of $\sigma_{j-1}$;
				\item the random variable $\widetilde{W}_{1,i}$ is independent of $\beta_{\sigma_i}$.
			\end{itemize}
		\end{itemize}
	\end{proposition}
	\begin{remark}
		We do not repeat the proof as it is already given in the previous reference. We highlight the idea which can be useful in the following paragraphs. Generally, using the branching embedding in Definition~\ref{def.ContinuousFrag}, the size of the cluster connected to the root is a Yule process of rate $p$, thus one obtains the characterisation of Mittag--Leffler distribution of $\widetilde{W}_{1}$ (see \cite[Theorem~3.1.(i)]{baur2015fragmentation}). We denote by $\widetilde{T_\infty^{i, \omega}}$ the subtree rooted at $i$ 
		\begin{align}\label{eq.defSubtree}
			\widetilde{T_\infty^{i, \omega}} := \{v \in \N_+: v \geq i  \text{ and is connected to } i \text{ on } \widetilde{T_\infty^{\omega}}\},
		\end{align}
		with $\widetilde{T_n^{i, \omega}}$ as its restriction on $[n]$. Then for general  $i \in \N_+$, the process $(\vert \widetilde{T_n^{i, \omega}}\vert)_{n \in \N_+}$ can be analysed similarly as $(\vert \widetilde{T_n^{1, \omega}}\vert)_{n \in \N_+}$ and we obtain(see \cite[Lemma~3.3]{baur2015fragmentation})
		\begin{align}\label{eq.SubtreeLimit}
			n^{-p} \vert \widetilde{T_n^{i, \omega}} \vert \xrightarrow[\text{a.s.}]{n \to \infty} \widetilde{W}'_{1,i} (\beta_{i})^p, \qquad \forall i \in \N_+.
		\end{align}
		Here $\widetilde{W}'_{1,i} \stackrel{d}{=} \widetilde{W}_{1}$ and $\beta_i$ is a Beta random variable  with parameters $(1, i-1)$. 
		We also notice the difference between the subtrees $(\widetilde{T_\infty^{i, \omega}})_{i \in \N_+}$ and the partition $(\widetilde{\Pi}^{(n)}_j)_{j \in \N_+}$. In particular, $\widetilde{T_\infty^{i, \omega}}$ is a cluster in the bond-percolation $\widetilde{T_\infty^{\omega}}$ if and only if its root $i$ is not connected to its parent, i.e. $\omega(i) = 0$ in our coupling from Lemma~\ref{lem.site_bond_coupling}. This explains where the random variables $(\sigma_i)_{i \in \N_+}$ come from and why they are independent of $(\widetilde{W}'_{1,j})_{j \in \N_+}$ and $(\beta_j)_{j \in \N_+}$. 
	\end{remark}

	To derive the scaling limit of $\Pi$, we employ the following corollary of Lemma~\ref{lem.site_bond_coupling}.
	\begin{corollary}\label{cor.BlockCoupling}
		Under the site-bond-percolation coupling in Lemma~\ref{lem.site_bond_coupling}, we define $(\widetilde{\Pi}_{i,j})_{i \in \N_+, j\in\N}$ such that 
		\begin{align}\label{eq.PartitionIsolation}
			\widetilde{\Pi}_i = \bigsqcup_{j=0}^\infty \widetilde{\Pi}_{i,j}.
		\end{align}
		Here $\widetilde{\Pi}_{i,j}$ are vertex sets of site-percolation clusters coming from the root-isolation for the cluster of $\widetilde{\Pi}_i$ in Lemma~\ref{lem.site_bond_coupling}, and $j$ is indexed by the increasing order of the root vertex.
		
		The partition $\Pi$ (associated to the site-percolation) can be obtained from  $\widetilde{\Pi}$ (associated to the bond-percolation) in the following way:
		\begin{itemize}
			\item If $\omega(1) = 1$, then $\{\Pi_i : i \in \N_+\} = \{\widetilde{\Pi}_{1}\} \cup \Ll\{\widetilde{\Pi}_{i,j} :i\geq 2, j\in\N\Rr\}$.
			\item If $\omega(1) = 0$, then $\{\Pi_i : i \in \N_+\} = \{\widetilde{\Pi}_{i,j}: i \in \N_+, j\in\N\}$.	
		\end{itemize}
	\end{corollary}
	\begin{proof}
		Here the equality in the statement is for the set, so the elements are equal by rearrangement. The proof is a direct translation from Lemma~\ref{lem.site_bond_coupling} and \eqref{eq.PartitionIsolation}. 
	\end{proof}
	
	Now we can state the scaling limit of $\Pi$. We follow the convention that $\widetilde{\Pi}^{(n)}_{i,j}$ stands for the restriction of $\widetilde{\Pi}_{i,j}$ on $[n]$.
	\begin{proposition}\label{prop.SiteClusterLimit}
		Under the site-bond-percolation coupling in Lemma~\ref{lem.site_bond_coupling}, there exist i.i.d.\ random variables $(U_{i,j})_{i,j\in\N_+}$ uniformly distributed in $[0,1]$, such that: 
		\begin{enumerate}
			\item If $\omega(1) = 0$, then for every $k \in \N_+$, there exists $(i,j) \in \N_+ \times \N$  such that 
			$\Pi_k = \widetilde{\Pi}_{i,j}$ (see Corollary~\ref{cor.BlockCoupling}), and we have
			\begin{align}\label{eq.SiteClusterLimit1}
				n^{-p} \vert \Pi^{(n)}_k \vert = n^{-p} \vert \widetilde{\Pi}^{(n)}_{i,j} \vert \xrightarrow[\text{a.s.}]{n \to \infty} \widetilde{W}_i V_{i,j} =: W_k,
			\end{align}
			where $V_{i,j}$ is defined as 
			\begin{align}\label{eq.defV}
				V_{i,0} := 0, \qquad	V_{i,j} := \Ll(\prod_{\ell = 1}^{j-1}(1-U_{i,\ell})\Rr) U_{i,j}.
			\end{align}
			
			\item If $\omega(1) = 1$, then $\Pi_1 = \widetilde{\Pi}_{1}$ and 
			\begin{align}\label{eq.SiteClusterLimit2}
				n^{-p} \vert \Pi^{(n)}_1 \vert = n^{-p} \vert \widetilde{\Pi}^{(n)}_{1} \vert \xrightarrow[\text{a.s.}]{n \to \infty} \widetilde{W}_1.
			\end{align}
			Moreover, the statement around \eqref{eq.SiteClusterLimit1} is valid for $k \geq 2$ and $i \geq 2, j \in \N$.
			
			Here $(\widetilde{W}_i)_{i \in \N_+}$ are the random variables defined in \eqref{eq.BondClusterLimit}, and $(U_{i,j})_{i,j\in\N_+}$ are independent of $(\widetilde{W}_i)_{i \in \N_+}$.
		\end{enumerate}
	\end{proposition}
	\begin{proof}
		Equation~\eqref{eq.SiteClusterLimit2} is the case without root-isolation, so \eqref{eq.BondClusterLimit} applies directly. For the cases with the root-isolation, we have the coordinate-wise convergence of the following
		\begin{align*}
			n^{-p} \Ll(\vert \widetilde{\Pi}^{(n)}_{i,0}\vert, \vert \widetilde{\Pi}^{(n)}_{i,1}\vert, \vert \widetilde{\Pi}^{(n)}_{i,2}\vert, \cdots \Rr) = n^{-p} \vert \widetilde{\Pi}^{(n)}_{i} \vert \Ll(\frac{\vert \widetilde{\Pi}^{(n)}_{i,0}\vert}{\vert \widetilde{\Pi}^{(n)}_{i} \vert}, \frac{\vert \widetilde{\Pi}^{(n)}_{i,1}\vert}{\vert \widetilde{\Pi}^{(n)}_{i} \vert}, \frac{\vert \widetilde{\Pi}^{(n)}_{i,2}\vert}{\vert \widetilde{\Pi}^{(n)}_{i} \vert}, \cdots \Rr).
		\end{align*} 
		The limit of $n^{-p} \vert \widetilde{\Pi}^{(n)}_{i} \vert$ has already been deduced in \eqref{eq.BondClusterLimit}. Notice that almost surely ${\lim_{n \to \infty } \vert \widetilde{\Pi}^{(n)}_{i}\vert = \infty}$ and $ \widetilde{\Pi}^{(n)}_{i,0} = \{i\}$, we have $\lim_{n \to \infty} \frac{\vert \widetilde{\Pi}^{(n)}_{i,0}\vert}{\vert \widetilde{\Pi}^{(n)}_{i} \vert} = 0$.  
		The convergence of the remaining coordinates concern the proportion of subtrees after removing the root of RRT, which is the convergence to the \textbf{uniform stick-breaking distribution} (see \cite[Section~3.1]{ptiman2002})
		\begin{align*}
			\Ll(\frac{\vert \widetilde{\Pi}^{(n)}_{i,1}\vert}{\vert \widetilde{\Pi}^{(n)}_{i} \vert}, \frac{\vert \widetilde{\Pi}^{(n)}_{i,2}\vert}{\vert \widetilde{\Pi}^{(n)}_{i} \vert}, \frac{\vert \widetilde{\Pi}^{(n)}_{i,3}\vert}{\vert \widetilde{\Pi}^{(n)}_{i} \vert}, \cdots \Rr) \xrightarrow[\text{a.s.}]{n \to \infty} (U_{i,1}, (1-U_{i,1})U_{i,2}, (1-U_{i,1})(1-U_{i,2})U_{i,3}, \cdots).
		\end{align*}
		The uniform stick-breaking distribution is also known as a special case of Griffiths--Engen--McCloskey (GEM) distribution (see \cite[Section~3.2]{ptiman2002}). Here $(U_{i,j})_{j \in \N_+}$ are i.i.d. random variables uniformly on $[0,1]$. Moreover, these random variables are independent for every root-isolation on every $\widetilde{\Pi}_{i}$, thus we prove our statement.
	\end{proof}

	\subsection{Moment estimate}
	Proposition~\ref{prop.SiteClusterLimit} gives us the scaling limit of the size for one cluster, which is a partial result of Theorem~\ref{thm.main2}. Thus, some improvement is needed. In this part, we prove the main technical estimate for the proof of Theorem~\ref{thm.main2} in the following lemma. It will provide us a useful dominated function in the convergence of the vectors in $\ell^q$. Recall the process $(\TT_t)_{t \geq 0}$ as the branching embedding of the RRT and $(\TT^\omega_t)_{t \geq 0}$ as its associated site-percolation in Definition~\ref{def.ContinuousFrag}. We denote by $\TT^{i,\omega}_t$ the subtree rooted by $i$ in $\TT^\omega_t$, i.e. 
	\begin{align}\label{eq.defSubtreeSite}
		\TT^{i,\omega}_t := \{v \in \N_+: v \geq i  \text{ and is connected to } i \text{ on } \TT^{\omega}_t\}.
	\end{align} 
	We also recall that $\mathbf{b}(i)$ stands for the birth time of the vertex $i$ in $(\TT_t)_{t \geq 0}$.
	
	\begin{lemma}[Moment estimate]\label{lem.Moment}
		For every $p \in (0,1)$ and $q \in (\frac{1}{p}, \infty)$, there exists a finite positive constant $C_{p,q}$ such that the following estimate holds for every $k \in \N_+$
		\begin{align}\label{eq.Moment}
			\E\Ll[\sum_{i \in \N_+: \mathbf{b}(i) \geq k} \Ll(\sup_{t \in \R_+} e^{-pt} \vert \TT^{i,\omega}_t\vert\Rr)^q\Rr] \leq C_{p,q} e^{-(pq - 1)k}.
		\end{align}
	\end{lemma}
	\begin{proof}
		We make the discretisation of the birth time
		\begin{align}\label{eq.MomentDiscrete}
			\E\Ll[\sum_{i \in \N_+: \mathbf{b}(i) \geq k} \Ll(\sup_{t \in \R_+} e^{-pt} \vert \TT^{i,\omega}_t\vert\Rr)^q\Rr] = \sum_{m=k}^\infty \E\Ll[\sum_{i \in \N_+: \mathbf{b}(i) \in [m,m+1)} \Ll(\sup_{t \in \R_+} e^{-pt} \vert \TT^{i,\omega}_t\vert\Rr)^q\Rr].
		\end{align}
		We just focus on one interval. For $t \geq \mathbf{b}(i)$, the term $\vert \TT^{i,\omega}_t\vert$ is a Yule process of density $p$ when this is an open cluster, or $\vert \TT^{i,\omega}_t\vert = 1$ when the vertex $i$ is closed, and every term evolves independently. Thus, we have a family of i.i.d.\ Yule processes $\{(G^i_t)_{t \geq 0}\}_{i \in \N_+}$ of parameter $p$ such that 
		\begin{align*}
			\forall i \in \N_+: \mathbf{b}(i) \in [m,m+1), \qquad e^{-pt} \vert \TT^{i,\omega}_t\vert &= e^{-p\mathbf{b}(i)} \Ll(e^{-p(t-\mathbf{b}(i))} \Ll(G^i_{t-\mathbf{b}(i)}\omega(i) + \Ll(1-\omega(i)\Rr)\Rr)\Rr)\\
			&\leq e^{-pm} \Ll(e^{-p(t-\mathbf{b}(i))} \Ll(G^i_{t-\mathbf{b}(i)}\Rr)\Rr).
		\end{align*}
		Here we gain a factor $e^{-pm}$ using the condition $\mathbf{b}(i) \in [m,m+1)$. This implies that
		\begin{align*}
			\forall i \in \N_+: \mathbf{b}(i) \in [m,m+1), \qquad \Ll(\sup_{t \in \R_+} e^{-pt} \vert \TT^{i,\omega}_t\vert\Rr) \leq e^{-pm}\Ll(\sup_{t \in \R_+} e^{-pt} G^i_t\Rr).
		\end{align*}
		We then apply the Wald identity to obtain 
		\begin{multline}\label{eq.MomentOneTerm}
			\E\Ll[\sum_{i \in \N_+: \mathbf{b}(i) \in [m,m+1)} \Ll(\sup_{t \in \R_+} e^{-pt} \vert \TT^{i,\omega}_t\vert\Rr)^q\Rr] \\
			\leq e^{-pqm} \E[\#\{i \in \N_+: \mathbf{b}(i) \in [m,m+1)\}]\E\Ll[\Ll(\sup_{t \in \R_+} e^{-pt} G_t\Rr)^q\Rr].
		\end{multline}
		Here we denote by $(G_t)_{t \geq 0}$ a canonical Yule process of density $p$. Then $(e^{-pt} G_t)_{t \geq 0}$ is a martingale satisfying that $e^{-pt} G_t \xrightarrow[a.s.]{t \to \infty} \mathcal{E}$, where $\mathcal{E}$ follows the standard exponential distribution (see \cite[Remark 5, Page 130]{AN2004}). Moreover,  Doob's maximal inequality applies that 
		\begin{align}\label{eq.Doob}
			\E\Ll[\Ll(\sup_{t \in \R_+} e^{-pt} G_t\Rr)^q\Rr] \leq \Ll(\frac{q}{q-1}\Rr)^q \E[ \mathcal{E}^q] = \Ll(\frac{q}{q-1}\Rr)^q \Gamma(q+1).
		\end{align}
		The number of vertices can be estimated by the size of Yule process
		\begin{align}\label{eq.numVertices}
			\E[\#\{i \in \N_+: \mathbf{b}(i) \in [m,m+1)\}] \leq \E[\#\{i \in \N_+: \mathbf{b}(i) < m+1\}] = e^{m+1}.
		\end{align}
		We combine \eqref{eq.numVertices} and \eqref{eq.Doob}, and put them back to \eqref{eq.MomentOneTerm} to obtain
		\begin{align}\label{eq.MomentOneTermBound}
			\E\Ll[\sum_{i \in \N_+: \mathbf{b}(i) \in [m,m+1)} \Ll(\sup_{t \in \R_+} e^{-pt} \vert \TT^{i,\omega}_t\vert\Rr)^q\Rr] \leq e\Gamma(q+1)\Ll(\frac{q}{q-1}\Rr)^q e^{-(pq-1)m}.
		\end{align}
		As $pq > 1$, this bound is of exponential decay in function of $m$. We put this estimate back to  \eqref{eq.Moment}, then obtain the desired result.
	\end{proof}

	\subsection{Proof of Theorem~\ref{thm.main2}}
	With the preparation in the previous two subsections, we are now ready to prove Theorem~\ref{thm.main2}. Some basic facts about the convergence in $\ell^q$ space are recalled in Proposition~\ref{prop.ellp} in Appendix~\ref{sec.ellp}.
	
	\begin{proof}[Proof of Theorem~\ref{thm.main2}]
		
		We denote by $\vert \Pi^{(n)} \vert = (\vert \Pi^{(n)}_i \vert)_{i \in \N_+}$ and $\vert \Pi^{(n),\downarrow} \vert = (\vert \Pi^{(n),\downarrow}_i \vert)_{i \in \N_+}$ as a shorthand of vectors. The natural candidate of the limit is the ranked version of $(W_i)_{i \in \N_+}$ by decreasing order obtained in Proposition~\ref{prop.SiteClusterLimit}, which is denoted by $W^{\downarrow} = (W^{\downarrow}_i)_{i \in \N_+}$. From (2) of Proposition~\ref{prop.ellp}, because the ranking rearrangement 
		is a non-expansive operator (see Proposition \ref{sec.ellp} for the definition), for every $q \in (\frac{1}{p}, \infty]$ we have  
		\begin{align}\label{eq.main2usual}
			\norm{n^{-p}\vert \Pi^{(n),\downarrow} \vert - W^{\downarrow}}_{\ell^q} \leq  \norm{n^{-p}\vert \Pi^{(n)} \vert - W}_{\ell^q}.
		\end{align}
		Then we only need to focus on the convergence of the right-hand side.

		To prove the $\ell^q$ convergence, we have an equivalent description in (3) of Proposition~\ref{prop.ellp}. Here the coordinate-wise convergence is already proved in Proposition~\ref{prop.SiteClusterLimit}, and it remains to verify the second condition for the tail, i.e.
		\begin{align}\label{eq.TailBlock}
			\lim_{k \to \infty} \limsup_{n \to \infty} \sum_{i \geq k} \Ll(n^{-p}\vert \Pi^{(n)}_i \vert\Rr)^q = 0, \qquad \P\text{-a.s.}.
		\end{align}
		Actually, one only needs to prove \eqref{eq.TailBlock} for $q = \frac{1}{p} + \epsilon$ with $\epsilon > 0$ and arbitrarily small. Then (1) of Proposition~\ref{prop.ellp} allows to extend the convergence to $(\frac{1}{q}, \infty]$.

		We will prove \eqref{eq.TailBlock} once again using the embedding in branching process from Definition~\ref{def.ContinuousFrag}. Recall  the stopping time $\tau_n$ defined in \eqref{eq.stopingTime} and $T_n^\omega = \mathcal{T}_{\tau_n}^\omega$ from \eqref{eq.correspondence}, then for every $k \in \N_+$  we have
		\begin{align}\label{eq.inclusion}
			\Ll\{\Pi^{(n)}_i: i \geq k\Rr\} \subset \Ll\{V(\TT^{i,\omega}_{\tau_n}): i \geq k\Rr\}.
		\end{align}
		Notice that every cluster $\Pi^{(n)}_i$ is necessarily a subtree in $\mathcal{T}_{\tau_n}^\omega$, and the its root vertex is the $i$-th smallest root, so it must be at least $i$ (see Definition~\ref{def.Canonical}). This implies the inclusion of sets above. Then we have
		\begin{align*}
			\sum_{i \geq k} \Ll(n^{-p}\vert \Pi^{(n)}_i \vert\Rr)^q \leq \sum_{i \geq k} \Ll(n^{-p}\vert \TT^{i,\omega}_{\tau_n}\vert\Rr)^q,
		\end{align*}
		and the limits of the two sides give us
		\begin{equation}\label{eq.TightReduction0}
			\begin{split}
				\limsup_{n \to \infty}\sum_{i \geq k} \Ll(n^{-p}\vert \Pi^{(n)}_i \vert\Rr)^q &\leq \limsup_{n \to \infty}\sum_{i \geq k} \Ll(n^{-p}\vert \TT^{i,\omega}_{\tau_n}\vert\Rr)^q\\
				&= \limsup_{n \to \infty} (e^{p \tau_n} n^{-p})^q \limsup_{n \to \infty}\sum_{i \geq k} \Ll(e^{-p \tau_n}\vert \TT^{i,\omega}_{\tau_n}\vert\Rr)^q\\
				&\leq \lim_{n \to \infty} e^{pq (\tau_n - \log n)} \limsup_{t \to \infty}\sum_{i \geq k} \Ll(e^{-p t}\vert \TT^{i,\omega}_{t}\vert\Rr)^q.
			\end{split}
		\end{equation}
		The quantity $(\tau_n - \log n)$ in the first term admits an almost sure finite limit, which is a classical result and see \cite[Theorem~3, Page~120]{AN2004} for its proof. Therefore, we only need to show the decay of the second term, which can be bounded by  
		\begin{equation}\label{eq.TightReduction}
			\begin{split}
				\lim_{k \to \infty} \limsup_{t \to \infty}\sum_{i \geq k} \Ll(e^{-p t}\vert \TT^{i,\omega}_{t}\vert\Rr)^q	& \leq \lim_{k \to \infty} \sum_{i \geq k} \Ll(\sup_{t \in \R_+} e^{-p t}\vert \TT^{i,\omega}_{t}\vert\Rr)^q\\
				&=  \lim_{k \to \infty}  \sum_{i \in \N_+: \mathbf{b}(i) \geq k} \Ll(\sup_{t \in \R_+} e^{-p t}\vert \TT^{i,\omega}_{t}\vert\Rr)^q.
			\end{split}
		\end{equation}
		and it suffices to prove the almost sure convergence of the right-hand side. This can be realized by the moment estimate in Lemma~\ref{lem.Moment}: we define the event $A_k$
		\begin{align*}
			A_k := \Ll\{  \sum_{i \in \N_+: \mathbf{b}(i) \geq k} \Ll(\sup_{t \in \R_+} e^{-p t}\vert \TT^{i,\omega}_{t}\vert\Rr)^q  > e^{\frac{-(pq - 1)k}{2}}\Rr\}.
		\end{align*}
		Then by Markov inequality, we have 
		\begin{align*}
			\sum_{k=1}^\infty\P[A_k] \leq \sum_{k=1}^\infty C_{p,q}e^{\frac{-(pq - 1)k}{2}} < \infty.
		\end{align*}
		Then Borel--Cantelli lemma applies, which implies that $(A_k)_{k \in \N_+}$ only happens finite times almost surely and $\lim_{k \to \infty}  \sum_{i \in \N_+: \mathbf{b}(i) \geq k} \Ll(\sup_{t \in \R_+} e^{-p t}\vert \TT^{i,\omega}_{t}\vert\Rr)^q = 0$. We put it back to \eqref{eq.TightReduction} and \eqref{eq.TightReduction0}, then this concludes \eqref{eq.TailBlock}. 
		
	\end{proof}
	
	\section{Further discussions}\label{sec.discussions}
	
	Despite the simple definition, a lot of questions can be asked for the site-percolation in this paper. Here we list some questions from the audience, during \textit{The 10th Workshop on Branching Processes and Related Topics} held at Shenzhen MSU-BIT University.	
	\begin{question}[by Xinxin Chen]
		Theorem~\ref{thm.main2} addresses the scaling limit of the $i$-th largest cluster $\Ll\vert \Pi^{(n), \downarrow}_i \Rr\vert$, and the index $i$ is fixed when $n$ (the size of RRT) goes to infinity. What should be the  order of $\Ll\vert \Pi^{(n), \downarrow}_i \Rr\vert$ for a general index ? For example, if we look at $\Ll\vert \Pi^{(n), \downarrow}_{\lfloor \theta n \rfloor}\Rr\vert$, or $\Ll\vert \Pi^{(n), \downarrow}_{\lfloor n^{\alpha}\rfloor}\Rr\vert$ with $\alpha \in (0,1)$, is there some possible non-trivial asymptotic ?
	\end{question}
	
	\begin{question}[by Michel Pain]
		Conditioned that the vertex $1$ is open, what is the probability that the largest cluster contains the index $1$ ? That is, 
		\begin{align*}
			\P\Ll[\Pi^{(n)}_{1} = \Pi^{(n), \downarrow}_1 \vert \omega(1)=1\Rr] = ?
		\end{align*}
		Heuristically, in the embedding branching process, the cluster containing $1$ will have more time to grow in this situation, and naturally has more chance to be the largest. We are interested in the probability above, especially for large $n$. This question can also be asked for the bond-percolation, where the conditional event $\{\omega(1)=1\}$ is not needed. 
	\end{question}

	\begin{question}[by Quan Shi]
		Bertoin considered the percolation on RRT in the supercritical regime in \cite{bertoin2014sizes}, i.e. $p(n) = 1 - s/ \ln n +o(1/\ln n)$ which goes to $1$ for large $n$. The present paper explores the critical regime as $p(n)$ is constant. Then, what should be a reasonable model for the  subcritical regime and what is its scaling limit of the largest cluster ? This question was also mentioned at the end of the introduction of \cite{bertoin2014sizes}.
	\end{question}
	
	\begin{question}[by Hui He]
		There are a lot of statistics about RRT, but very few results about its scaling limit of geometry. One reason is that RRT is too fat: when the size is $n$, the typical height is $\log n$. However, is it possible to derive a scaling limit conditioned on its height or width ?
	\end{question}
	
	Finally, we reformulate another question in the anonymous report.
	
	\begin{question}[by an anonymous referee]
		Consider a model in the spirit of preferential attachment: let $\mathrm{deg}(v)$ stand for the degree of the vertex $v$, then we assign the site open in function of $\mathrm{deg}(v)$, how will it influence Theorem~\ref{thm.main1} and ~\ref{thm.main2} ?
	\end{question}

	\appendix
	
	\section{Convergence in $\ell^p$ space}\label{sec.ellp}
	We recall some basic properties in $\ell^p$ space.
	\begin{proposition}\label{prop.ellp}
		For any $p \in [1,\infty]$, the following properties hold for the $\ell^p$ space.
		\begin{enumerate}
			\item For every $1 \leq p < q \leq \infty$ and a sequence $\mathbf{x} = (x_i)_{i \in \N_+}$, then $\norm{\mathbf{x}}_{\ell^q} \leq \norm{\mathbf{x}}_{\ell^p}$.
			\item (Nonexpansivity of rearrangement) For every two positive sequences $\mathbf{x}, \mathbf{y} \in \ell^p$, the following estimate holds
			\begin{align*}
				\norm{\mathbf{x}^{\downarrow} - \mathbf{y}^{\downarrow}}_{\ell^p} \leq \norm{\mathbf{x} - \mathbf{y}}_{\ell^p}.
			\end{align*}
			
			\item Given a family of sequence $\mathbf{x}_n = (x_{n,1}, x_{n,2}, \cdots) \in \ell^p$, then $(\mathbf{x}_n)_{n \in \N_+}$ converges in $\ell^p$ if and only if the following two conditions are valid:
			\begin{itemize}
				\item For every coordinate $i \in \N_+$, the sequence $(x_{n,i})_{n \in \N_+}$ converges.
				\item Let $\mathbf{x}^{(k)}_n = (x_{n,k+1}, x_{n,k+2}, \cdots)$ and $\lim_{k \to \infty} \limsup_{n \to \infty} \norm{\mathbf{x}^{(k)}_n}_{\ell^p} = 0$. 
			\end{itemize}
		\end{enumerate}
	\end{proposition}
	\begin{proof}
		For (1), the case $q = \infty$ is clear and the case $\norm{\mathbf{x}}_{\ell^p} = \infty$ is trivial, thus we focus on the case $p,q \in [1, \infty)$ and $\norm{\mathbf{x}}_{\ell^p} < \infty$. By  normalisation, we suppose without loss of generality  $\norm{\mathbf{x}}_{\ell^p} = 1$. Then $\vert x_{n,i}\vert \leq 1$ for every $i \in \N_+$, and thus $\vert x_{n,i}\vert^q \leq \vert x_{n,i}\vert^p$ because $q > p$. Therefore, we obtain
		\begin{align*}
			\norm{\mathbf{x}}^q_{\ell^q} = \sum_{i=1}^\infty\vert x_i\vert^q \leq \sum_{i=1}^\infty\vert x_i\vert^p = 1 = \norm{\mathbf{x}}^q_{\ell^p}.
		\end{align*}  
		This justifies the inequality.
		
		(2) is classical and we cite \cite[Theorem~3.5]{lieb2001} for its proof.

		Concerning (3), we suppose at first that $\mathbf{x}_n \xrightarrow[\ell^p]{n\to \infty} \mathbf{x}$, then it is obvious that the $\ell^p$ convergence implies the coordinate-wise convergence. Moreover, we notice that $\norm{\mathbf{x}^{(k)}_n - \mathbf{x}^{(k)}}_{\ell^p} \leq \norm{\mathbf{x}_n - \mathbf{x}}_{\ell^p}$, so the convergence of $\mathbf{x}_n$ in $\ell^p$ also implies that of $\mathbf{x}^{(k)}_n$, and we verify the second condition
		\begin{align*}
			\limsup_{n \to \infty} \norm{\mathbf{x}^{(k)}_n}_{\ell^p} =  \norm{\mathbf{x}^{(k)}}_{\ell^p} \xrightarrow{k \to \infty} 0.
		\end{align*}
		
		For the other direction, we suppose now $x_i := \lim_{n \to \infty} x_{n,i}$ and $\mathbf{x} := (x_i)_{i \in \N_+}$. It gives us a natural candidate and we prove at first $\mathbf{x} \in \ell^p$. Picking an arbitrary $k \in \N_+$,  we apply Fatou's lemma
		\begin{align*}
			\norm{\mathbf{x}}^p_{\ell^p}  \leq \liminf_{n \to \infty} \norm{\mathbf{x}_n}^p_{\ell^p} = \liminf_{n \to \infty}\sum_{i=1}^\infty \vert x_{n,i}\vert^p 	\leq \sum_{i=1}^k \vert x_{i}\vert^p + \limsup_{n \to \infty} \Ll(\sum_{i=k}^\infty \vert x_{n,i}\vert^p\Rr) < \infty.
		\end{align*}
		In the last two inequalities above, we apply the two conditions. The $\ell^p$ convergence can be verified similarly
		\begin{align*}
			\norm{\mathbf{x}_n - \mathbf{x}}^p_{\ell^p} &\leq \sum_{i=1}^{k} \vert x_{n,i} - x_i\vert^p + \norm{\mathbf{x}^{(k)}_n - \mathbf{x}^{(k)}}^p_{\ell^p} \\
			&\leq \sum_{i=1}^{k} \vert x_{n,i} - x_i\vert^p + C_p\norm{\mathbf{x}^{(k)}_n}^p_{\ell^p} + C_p\norm{\mathbf{x}^{(k)}}^p_{\ell^p}.
		\end{align*}
		We use the first condition of coordinate-wise convergence for the finite dimension and obtain
		\begin{align*}
			\limsup_{n \to \infty}\norm{\mathbf{x}_n - \mathbf{x}}^p_{\ell^p} &\leq \limsup_{n \to \infty} \Ll(\sum_{i=1}^{k} \vert x_{n,i} - x_i\vert^p + C_p\norm{\mathbf{x}^{(k)}_n}^p_{\ell^p} + C_p\norm{\mathbf{x}^{(k)}}^p_{\ell^p}\Rr)\\
			&=  \limsup_{n \to \infty} 2 C_p\norm{\mathbf{x}^{(k)}_n}^p_{\ell^p}.
		\end{align*}
		Here we also use Fatou's lemma in the last inequality that $\norm{\mathbf{x}^{(k)}}^p_{\ell^p} \leq \liminf_{n \to \infty} \norm{\mathbf{x}^{(k)}_n}^p_{\ell^p}$. Since $k$ is arbitrary, we take the limit and use the second condition
		\begin{align*}
			\limsup_{n \to \infty}\norm{\mathbf{x}_n - \mathbf{x}}^p_{\ell^p} \leq \lim_{k \to \infty} \limsup_{n \to \infty} 2 C_p\norm{\mathbf{x}^{(k)}_n}^p_{\ell^p} = 0.
		\end{align*}  
		This concludes the proof.
	\end{proof}

	\textbf{Acknowledgement.} 
	C. Gu is supported by the National Key R\&D Program of China (No. 2021YFA1002700) and NSFC (No. 12301166). We thank Vincent Bansaye for inspiring discussions. We also thank the referees for careful reading and helpful comments.
	
	\bibliographystyle{abbrv}
	\bibliography{RRTRef}
\end{document}